\documentclass[a4paper, 11pt]{amsart}
\usepackage[lmargin=1.3in,rmargin=1.3in,tmargin=1.5in,bmargin=1.3in]{geometry}
\usepackage{graphicx}
\usepackage{amsmath, amssymb, amsthm, amscd}
\usepackage{mathtools}
\usepackage{hyperref}

\DeclareMathOperator\coker{coker}

\theoremstyle{plain}
\newtheorem{theorem}{Theorem} [section]
\newtheorem{lemma}[theorem]{Lemma}
\newtheorem{proposition}[theorem]{Proposition}
\newtheorem{corollary}[theorem]{Corollary}

\newtheorem{question}[theorem]{Question}

\newtheorem*{Donaldson obstruction}{Donaldson obstruction}
\newtheorem*{Correction term obstruction}{Correction term obstruction}

\theoremstyle{definition}

\theoremstyle{remark}
\newtheorem*{remark}{Remark}
\newtheorem{example}[theorem]{Example}

\newcommand{\Q}{\mathbb{Q}}
\newcommand{\Z}{\mathbb{Z}}

\newcommand{\rhcg}{\Theta^3_\mathbb{Q}}

\newcommand{\spinc}{\text{Spin}^c}

\title[Spherical 3-manifolds bounding rational homology balls]{Spherical 3-manifolds bounding rational homology balls}

\author{Dong Heon Choe}
\address{National Institute for Mathematical Sciences, 105 Gwanggyo-ro, Yeongtong-gu, Suwon-si, Gyeonggi-do 16229, Republic of Korea}
\email{dhchoe@nims.re.kr}

\author{Kyungbae Park}
\address{Department of Mathematical Sciences, Seoul National University, 1 Gwankak-ro, Gwanak-gu, Seoul 08826, Republic of Korea}
\email{kyungbaepark@snu.ac.kr}


\begin{document}
\begin{abstract}
	We give a complete classification of the spherical 3-manifolds that bound smooth rational homology 4-balls. Furthermore, we determine the order of spherical 3-manifolds in the rational homology cobordism group of rational homology 3-spheres. To this end, we use constraints for 3-manifolds to bound rational homology balls induced from Donaldson's diagonalization theorem and Heegaard Floer correction terms. 
\end{abstract}
 
\maketitle

\section{Introduction}
In this paper, every manifold is assumed to be compact and smooth unless stated otherwise. We say a 3-manifold (resp. a 4-manifold) is a \emph{rational homology 3-sphere} (resp. a \emph{rational homology 4-ball}) if the homology groups of the manifold are isomorphic to those of the 3-sphere (resp. the 4-ball) with coefficients in the rationals.  An interesting question in low-dimensional topology, which was proposed by Casson and appeared in Kirby's 1993 problem list \cite[Problem 4.5]{Kirby:1997-1}, asks which rational homology 3-spheres bound rational homology 4-balls; in other words, which of them admit rational homology ball fillings.

One significance of this question is given from a relation to the study of knot concordance. It is well known that if a knot $K$ is slice, then the 3-manifold $\Sigma(K)$ obtained by the double covering of $S^3$ branched along the knot bounds a rational homology ball; see \cite[Lemma 2]{Casson-Gordon:1986-1} for example. This gives an effective obstruction to a knot being slice. For example, using this condition as a main obstruction to a knot being slice, the celebrated slice-ribbon conjecture, posed by Fox \cite{Fox:1962-1}, \cite[Problem 1.33]{Kirby:1997-1}, has been established for some classes of knots such as all 2-bridge knots \cite{Lisca:2007-1}, and certain infinite families of pretzel knots \cite{Greene-Jabuka:2011-1, Lecuona:2015} and Montesinos knots \cite{Lecuona:2012-1}.

Another implication of this question arises from Fintushel and Stern's rational blow-down operation \cite{Fintushel-Stern:1997-1}, a method of obtaining a new 4-manifold by removing a certain submanifold in a 4-manifold and gluing a rational homology ball along the boundary. It has been at the heart of many interesting constructions of smooth and symplectic 4-manifolds \cite{Park:2005-1, Stipsicz-Szabo:2005-1, Park-Stipsicz-Szabo:2005-1}, and complex surfaces \cite{Lee-Park:2007-1, Park-Park-Shin:2009-1, Park-Park-Shin:2009-2}. The operation, which originally utilized lens spaces bounding rational homology balls, can be generalized further to use other 3-manifolds such as links of surface singularities with the same properties; see \cite{Stipsicz-Szabo-Wahl:2008-1} for example. Hence it is a fundamental task for this construction to determine which 3-manifolds admit rational homology ball fillings.

Earlier examples of rational homology 3-spheres bounding rational homology balls can be found in \cite{Akbulut-Kirby:1979-1, Casson-Harer:1981-1, Fintushel-Stern:1981-1, Fickle:1984-1}. An attempt to classify such manifolds among a specific family of 3-manifolds was made by Lisca. In \cite{Lisca:2007-1, Lisca:2007-2}, Lisca used Donaldson's diagonalization theorem to completely determine the lens spaces and the sums of lens spaces that bound rational homology balls. Following Lisca's results, Lecuona classified those 3-manifolds in some families of Seifert fibered manifolds \cite{Lecuona:2012-1, Lecuona:2017-1}. In \cite{Aceto-Golla:2017-1} Aceto and Golla studied the question for 3-manifolds obtained by Dehn surgery along a knot.

The goal of this article is to give a complete answer to the question for spherical 3-manifolds. A closed orientable 3-manifold is called \emph{spherical} if it admits a complete metric of constant curvature $+1$. A spherical 3-manifold can be also given by a quotient manifold of the form $S^3/\Gamma$, where $\Gamma$ is a finite subgroup of $SO(4)$ acting freely by the rotation on $S^3$. Notice that any spherical 3-manifold admits finite fundamental group, which is isomorphic to $\Gamma$. Conversely, by Perelman's elliptization theorem, any closed orientable prime 3-manifold with finite fundamental group is spherical. In terms of the singularity theory in algebraic geometry, the family of spherical 3-manifolds can be identified with homeomorphism types of links of quotient surface singularities. 

Spherical 3-manifolds are divided into the following five types with regard to their fundamental groups: $\mathbf{C}$ (cyclic), $\mathbf{D}$ (dihedral), $\mathbf{T}$ (tetrahedral), $\mathbf{O}$ (octahedral) and $\mathbf{I}$ (icosahedral) type. The $\mathbf{C}$-type manifolds are lens spaces, and a complete classification of lens spaces bounding rational homology balls was given by Lisca as mentioned above. Let $\mathcal{R}$ be the subset of the rationals defined in \cite[Definition 1.1]{Lisca:2007-1}, so that the lens space $L(p,q)$ such that $p>q>0$ and $(p,q)=1$ bounds a rational homology ball if and only if $\frac{p}{q}\in\mathcal{R}$. The following is our main result.

\begin{theorem}\label{thm:main}
	A spherical 3-manifold $Y$ bounds a smooth rational homology 4-ball if and only if $Y$ or $-Y$ is homeomorphic to one of the following manifolds: 
	\begin{itemize}
		\item $L(p,q)$ such that $\frac{p}{q}\in\mathcal{R}$,
		\item $D(p,q)$ such that $\frac{p-q}{q'}\in\mathcal{R}$,
		\item $T_3$, $T_{27}$ and $I_{49}$,
	\end{itemize}
	where $p$ and $q$ are relative prime integers such that $p>q>0$, and $0<q'<p-q$ is the reduction of $q$ modulo $p-q$. For the notations of spherical manifolds, see Section \ref{sec:spherical}.
\end{theorem}

In fact, we study a more general question than the above in the context of homology cobordisms. We say two rational homology 3-spheres $Y_1$ and $Y_2$ are \emph{rational homology cobordant} if there is a smooth cobordism $W$ from $Y_1$ to $Y_2$ such that the inclusion maps $Y_1\hookrightarrow W\hookleftarrow Y_2$ induce isomorphisms between the homology groups with coefficients in the rationals. Up to this cobordism relation, the set of rational homology 3-spheres forms an abelian group, called the \emph{rational homology cobordism group}, denoted by $\Theta^3_\Q$. Our result, Theorem \ref{thm:main}, completely classifies spherical 3-manifolds representing the identity or of order one element in $\Theta^3_\Q$. Hence it is natural to ask more generally what the order of spherical manifolds are. The order of all $\mathbf{C}$-type manifolds (lens spaces) are also known by Lisca in \cite[Corollary 1.2]{Lisca:2007-2}. In this paper, we completely determine the order of all non-cyclic spherical manifolds.

\begin{theorem}\label{thm:order}
	The order of a spherical manifold of type $\mathbf{D}$, $D(p,q)$, in $\rhcg$ is the same as that of the lens space $L(p-q,q)$. The order of a spherical manifold $Y$ of type $\mathbf{T}$, $\mathbf{O}$ or $\mathbf{I}$ is given as:
	\begin{itemize}
		\item $1$ if and only if $\pm Y\cong T_{3},T_{27}$ or $I_{49}$,
		\item $2$ if and only if $\pm Y\cong T_{15}$, 
		\item $\infty$ otherwise.
	\end{itemize}
\end{theorem}

Note that not much is known about the structure of the group $\rhcg$; see \cite{Kim-Livingston:2014-1} for example. In particular, it has been conjectured that any nontrivial torsion element in $\rhcg$ would only appear as two torsion. Our theorem supports the conjecture in the case of spherical manifolds.

We sketch the idea of the proof of our results. All $\mathbf{D}$-type manifolds and a certain family of $\mathbf{T}$-type manifolds (more precisely manifolds of the form $T_{6(b-2)+3}$) are contained in the class of manifolds considered by Lecuona in \cite{Lecuona:2012-1, Lecuona:2017-1}, namely Seifert fibered rational homology 3-spheres with complementary legs. Lecuona showed that any manifold in this class is rational homology cobordant to a lens space (including $S^3$), and hence the order in $\rhcg$ of the manifold is identical to that of the corresponding lens space.

For manifolds of type $\mathbf{T}$ (other than $T_{6(b-2)+3}$), $\mathbf{O}$ and $\mathbf{I}$, we use constraints for 3-manifolds to bound rational homology balls, coming from Donaldson's diagonalization theorem and Heegaard Floer correction terms. The argument combining the information from them has been used in many recent articles, see for example \cite{Greene-Jabuka:2011-1,Lecuona:2015,Aceto-Golla:2017-1,Choe-Park:2018-1}. 

Let $Y$ be a rational homology 3-sphere. Suppose there is a negative definite 4-manifold $X$ with boundary $Y$. If $Y$ bounds a rational homology ball $W$, then $X\cup_Y-W$ is a closed negative definite 4-manifold. Then Donaldson's diagonalization theorem \cite{Donaldson:1983, Donaldson:1987} implies that the intersection form of $X\cup_Y-W$ is diagonalizable over $\Z$, i.e. isometric to the standard diagonal lattice $(\Z^{b_2(X)},\langle-1\rangle^{b_2(X)})$. In particular, the intersection form of $X$, $Q_X$, embeds into $(\Z^{b_2(X)},\langle-1\rangle^{b_2(X)})$. Therefore, one might show that $Y$ does not admit a rational ball filling by proving the non-existence of an embedding of $Q_X$ into the standard diagonal lattice of the same rank. This strategy was successful enough to determine all $\mathbf{C}$-type manifolds bounding rational homology balls by Lisca and is also applicable to some other non-cyclic spherical manifolds. However, you will see that this constraint is not sufficient for some families of spherical manifolds. See Section \ref{sec:I_19}.

In order to resolve this, we shall employ another condition coming from Heegaard Floer theory. Heegaard Floer correction terms are rational valued invariants, denoted by  $d(Y,\mathfrak{t})$, associated to a rational homology 3-sphere $Y$ equipped with a spin$^c$ structure $\mathfrak{t}$ over $Y$, introduced by Ozsv\'ath and Szab\'o. In \cite{Ozsvath-Szabo:2003-2} it was shown that if $Y$ bounds a rational ball $W$, then $d(Y,\mathfrak{t})$ vanishes for any spin$^c$-structure $\mathfrak{t}$ that extends over $W$. Hence one can argue that a rational homology sphere $Y$ cannot bound a rational homology ball by observing the existence of non-vanishing correction term for a spin$^c$ structure over $Y$ that extends over any rational homology ball with the boundary $Y$.  

It turns out that most spherical manifolds of types $\mathbf{T}$ (other than manifolds of the form $T_{6(b-2)+3}$), $\mathbf{O}$ and $\mathbf{I}$, except the manifold $I_{49}$, can be obstructed from having finite order in $\rhcg$, by applying various techniques based on these two conditions.  For the manifold $I_{49}$, we explicitly find a rational homology 4-ball bounded by it to complete our classifications. We present Table \ref{tab:order} below to introduce the notations of each type of spherical manifold in this paper and organize our results.

\begin{table}[t!]
	\label{tab:order}
	{\renewcommand{\arraystretch}{1.2}
		\begin{tabular}{cclp{2.7cm}}
			Type& Spherical manifolds&Order in $\Theta^3_\Q$&Proved in\\
			\hline
			$\mathbf{C}$&$L(p,q)\cong Y(-1;(q,p-q))$& $1$ if $\frac{p}{q}\in \mathcal{R}$&Lisca \cite{Lisca:2007-2}\\
			&&2 if $\frac{p}{q}\in(\mathcal{S}\setminus\mathcal{R})\cup\mathcal{F}_2$&\\
			&&$\infty$ if $\frac{p}{q}\notin\mathcal{S}\cup\mathcal{R}\cup\mathcal{F}_2$&\\
			\hline
			$\mathbf{D}$&$D(p,q)\cong Y(1;(2,1),(2,1),(q,q-p))$&same as $L(p-q,q)$&Lecuona \cite{Lecuona:2017-1} \hspace{0.5cm}(Section \ref{sec:type_D})\\
			\hline
			$\mathbf{T}$&$T_{6(b-2)+1}\cong Y(b;(2,1),(3,2),(3,2))$ &$\infty$ for any $b$&Section \ref{sec:mu_bar}\\
			\cline{2-4}
			&$T_{6(b-2)+3}\cong Y(b;(2,1),(3,1),(3,2))$& $1$ if $b=2,6$&Section \ref{sec:type_T}\\
			&&$2$ if $b=4$&\\
			&&$\infty$ otherwise&\\
			\cline{2-4}
			&$T_{6(b-2)+5}\cong Y(b;(2,1),(3,1),(3,1))$&$\infty$ for any $b$&Section \ref{sec:mu_bar}\\
			\hline 
			$\mathbf{O}$&$O_{12(b-2)+1}\cong Y(b;(2,1),(3,2),(4,3))$&$\infty$ for any $b$ &Section \ref{sec:Donaldson_1}\\
			&$O_{12(b-2)+5}\cong Y(b;(2,1),(3,1),(4,3))$&$\infty$ for any $b$ &Section \ref{sec:Donaldson_1}\\
			&$O_{12(b-2)+7}\cong Y(b;(2,1),(3,2),(4,1))$&$\infty$ for any $b$ &Section \ref{sec:Donaldson_1}\\
			&$O_{12(b-2)+11}\cong Y(b;(2,1),(3,1),(4,1))$&$\infty$ for any $b>5$&Section \ref{sec:Donaldson_2}\\
			&&$\infty$ for $b=2,3,4,5$ &Section \ref{sec:d_revisited}\\
			\hline
			$\mathbf{I}$&$I_{30(b-2)+1}\cong Y(b;(2,1),(3,2),(5,4))$&$\infty$ for any $b$&Section \ref{sec:mu_bar}\\
			\cline{2-4}
			&$I_{30(b-2)+7}\cong Y(b;(2,1),(3,2),(5,3))$	&$\infty$ for any $b$&Section \ref{sec:mu_bar}\\
			\cline{2-4}
			&$I_{30(b-2)+11}\cong Y(b;(2,1),(3,1),(5,4))$	&$\infty$ for any $b$&Section \ref{sec:mu_bar}\\
			\cline{2-4}
			&$I_{30(b-2)+13}\cong Y(b;(2,1),(3,2),(5,2))$	&$\infty$ for any $b$&Section \ref{sec:mu_bar}\\
			\cline{2-4}
			&$I_{30(b-2)+17}\cong Y(b;(2,1),(3,1),(5,3))$	&$\infty$ for any odd $b$&Section \ref{sec:mu_bar}\\
			&&$\infty$ for even $b>7$&Section \ref{sec:Donaldson_2}\\
			&& $\infty$ for $b=2,4,6$&Section \ref{sec:d_revisited}\\
			\cline{2-4}
			&$I_{30(b-2)+19}\cong Y(b;(2,1),(3,2),(5,1))$&$\infty$ for any even $b$&Section \ref{sec:mu_bar}\\
			&&$\infty$ for any odd $b>5$&Section \ref{sec:Greene-Jabuka}\\
			&&$1$ if $b=3$&Proposition \ref{prop:ribbon}\\
			&&$\infty$ if $b=5$&Section \ref{sec:d_revisited}\\
			\cline{2-4}
			&$I_{30(b-2)+23}\cong Y(b;(2,1),(3,1),(5,2))$&$\infty$ for even $b$&Section \ref{sec:mu_bar}\\
			&& $\infty$ for odd $b>7$&Section \ref{sec:Donaldson_2}\\
			&& $\infty$ for $b=3,5,7$&Section \ref{sec:d_revisited}\\
			\cline{2-4}
			&$I_{30(b-2)+29}\cong Y(b;(2,1),(3,1),(5,1))$&$\infty$ for any $b$&Section \ref{sec:mu_bar}\\
			\hline
		\end{tabular}
	}
	\caption{The order of spherical manifolds in $\Theta^3_\Q$, where $p>q>0$ are relative prime integers and $b\geq2$.}
\end{table}

Spherical 3-manifolds (more generally Seifert fibered rational homology spheres) can be obtained as the double covering of $S^3$ branched along Montesinos links. Let $\mathcal{S}$ be the set of Montesinos knots that admit spherical branched double covers. Note that the slice-ribbon conjecture and the non-existence of nontrivial torsion elements of order different from two are well-known open questions in the study of knot concordance. We have the following by-products, which answer these question for Montesinos knots in $\mathcal{S}$. 

\begin{corollary}\label{cor:slice-ribbon}
	The slice-ribbon conjecture is true for knots in $\mathcal{S}$.
\end{corollary}

\begin{corollary}\label{cor:concordance_order}
	For a Montesinos knot $K\in \mathcal{S}$, the order of $K$ in the knot concordance group is equal to that of $\Sigma(K)$ in $\rhcg$.
\end{corollary}

\begin{remark}
	The concordance order of a knot $K$ is not always identical to the order of its branched double cover $\Sigma(K)$ in $\rhcg$, in general. For instance, it is known that, for odd $q>0$, the Brieskorn sphere $\Sigma(2,2q-1,2q+1)$, which is the double cover of $S^3$ branched along the $(2q-1,2q+1)$-torus knot, bounds a rational homology ball (in fact, a contractible ball) \cite{Casson-Harer:1981-1}. Whereas, any nontrivial torus knots are not slice; see \cite{Litherland:1979} for example. We also observe that this correspondence does not hold even for some Montesinos knots. For example, the Brieskorn sphere $\Sigma(2,3,7)$ is known to bound a rational homology ball \cite{Fintushel-Stern:1984-1}. However, since $\Sigma(2,3,7)$ has non-trivial Rohlin invariant, any double branch set of it cannot be slice. Note that $\Sigma(2,3,7)$ admits a natural branch set, the Montesinos knot $M(0;(2,-1),(3,1),(7,1))$.
\end{remark}

\subsection*{Organization} In the next section, we recall basic notions and properties of Seifert fibered manifolds and discuss how spherical manifolds can be interpreted as Seifert manifolds and Dehn surgery manifolds. In Section \ref{sec:obstructions}, we introduce our main obstructions: one from Donaldson's theorem and the other from Heegaard Floer correction terms, and we also discuss the computation of correction terms for spherical manifolds. In Section \ref{sec:Lecuona}, we recall Lecuona's result on Seifert manifolds with complementary legs and determine the order of manifolds of type $\mathbf{D}$ and of the form $T_{6(b-2)+3}$. In Section \ref{sec:rational_ball}, all spherical manifolds bounding rational homology balls are classified. Finally, in the last section, we determine the order of all spherical manifolds in $\rhcg$ and give the proofs of Corollary \ref{cor:slice-ribbon} and \ref{cor:concordance_order}.

\section{Spherical manifolds as Seifert manifolds and Dehn surgery manifolds}\label{sec:preliminary}
\subsection{Seifert fibered rational homology 3-spheres}\label{sec:Seifert_manifolds}
Spherical 3-manifolds are included in a broad class of 3-manifolds, Seifert fibered rational homology 3-spheres, or Seifert manifolds for short in this paper. We briefly recall some notions and properties of Seifert manifolds. See \cite{Orlik:1972} and \cite{Neumann-Raymond:1978-1} for more detailed expositions.

A Seifert manifold can be represented by the surgery diagram depicted in Figure \ref{fig:Seifert}. The collection of integers in the diagram,
\[(b;(\alpha_1,\beta_1),(\alpha_2,\beta_2),\dots,(\alpha_r,\beta_r)),\]
of which $b\in\Z$, and $\alpha_i>0$ and $\beta_i$ are coprime integers, is called the \emph{Seifert invariant}. Let $Y(b;(\alpha_1,\beta_1),\dots,(\alpha_r,\beta_r))$ denote the Seifert manifold corresponding to the invariant. It is clear from the surgery diagram that permutations among $(\alpha_i,\beta_i)$'s in an invariant result in the same 3-manifold. By the blow-down procedure and Rolfsen's twist respectively, we have the following homeomorphisms,
\[Y(b;(\alpha_1,\beta_1),\dots,(\alpha_r,\beta_r),(1,\pm1))\cong Y(b\mp1;(\alpha_1,\beta_1),\dots,(\alpha_r,\beta_r)),\]
and
\[Y(b;(\alpha_1,\beta_1),\dots,(\alpha_r,\beta_r))\cong Y(b-n;(\alpha_1,\beta_1),\dots,(\alpha_r,\beta_r-n\alpha_r)).\] 
See \cite{Gompf-Stipsicz:1999-1} for example. Hence any Seifert invariant can be normalized so that $b\in\Z^+$ and $\alpha_i>\beta_i>0$ are coprime integers to represent the same manifold, called the \emph{normalized Seifert invariant}. From now on, all Seifert invariants are assumed to be normalized unless stated otherwise. Any Seifert manifolds with $r\leq2$ are homeomorphic to lens spaces, and (orientation preserving) homeomorphism classes of lens spaces are classified as follow:
\[L(p,q)\cong L(p',q')\text{ if and only if }p=p'\text{ and }q\equiv (q')^{\pm1}\text{ modulo } p.\]
In this paper, $L(p,q)$ denotes the lens space obtained by $(p/q)$-surgery along the unknot. By the homeomorphism classification of Seifert manifolds (See \cite[Chapter 5]{Orlik:1972} for example), homeomorphism types of Seifert manifolds with $r\geq3$ are classified by their normalized invariants up to the permutations among the pairs $(\alpha_i,\beta_i)$'s. 
 
\begin{figure}
	\includegraphics[width=0.9\textwidth]{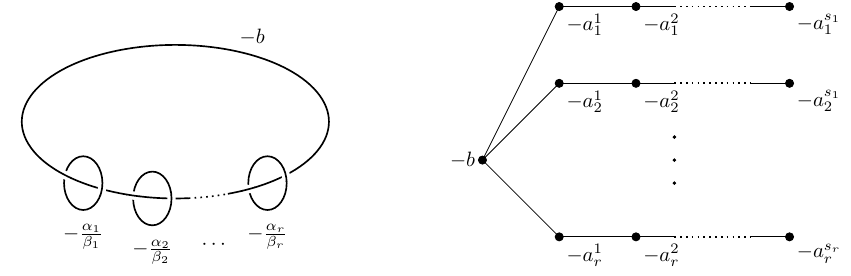}
	\caption{A surgery diagram of the Seifert manifold with the invariant $(b;(\alpha_1,\beta_1),\dots,(\alpha_r,\beta_r))$ and its associated plumbing graph.}
	\label{fig:Seifert}
\end{figure}

Let $Y$ be a Seifert manifold with the invariant $(b;(\alpha_1,\beta_1),\dots,(\alpha_r,\beta_r))$. We define \[e(Y)\vcentcolon=b-\sum_{i=1}^r\frac{\beta_i}{\alpha_i},\] called the \emph{Euler number} of a Seifert manifold. Note that the Euler number is independent of the choices of Seifert invariants for $Y$. It is easy to check that $Y$ is a rational homology 3-sphere if and only if $e(Y)\neq0$. From the linking matrix of the surgery diagram of $Y$, the order of $H_1(Y,\Z)$ is given as 
\begin{equation*}
	\left|H_1(Y,\Z)\right|=\alpha_1\alpha_2\dots\alpha_r\left|b-\sum_{i=1}^r\frac{\beta_i}{\alpha_i}\right|.
\end{equation*}

By changing the orientation of $Y$ if necessary, we may assume that $e(Y)>0$; note that $e(-Y)=-e(Y)$. Then there is a natural negative definite 4-manifold with the boundary $Y$ constructed as follows. Extend each pair $(\alpha_i,\beta_i)$ in the invariant by the following Hirzebruch-Jung continued fraction: 
\[\frac{\alpha_i}{\beta_i}=a^1_i-\cfrac{1}{a^2_i-\cfrac{1}{\dots-\cfrac{1}{a^{s_i}_i}}}
={\colon}[a^1_i,a^2_i,\dots,a^{s_i}_i],\]
where $a_i^j\geq 2$. The \emph{canonical (negative) definite 4-manifold} of $Y$ is the plumbed 4-manifold corresponding to the weighted graph in Figure \ref{fig:Seifert}, i.e. each vertex with weight $n$ represents the disk bundle over the 2-sphere with the Euler number $n$ and each edge connecting two vertices represents the plumbing of the associated two disk bundles: see \cite[Example 4.6.2]{Gompf-Stipsicz:1999-1} for the plumbing construction. The assumption that $e(Y)>0$ implies that the intersection form of the 4-manifold is negative definite.

\subsection{Spherical 3-manifolds as Seifert manifolds}\label{sec:spherical}
All spherical manifolds are known to be Seifert manifolds. A finite subgroup $\Gamma$ of $SO(4)$ acting freely on $S^3$ is either a cyclic group or a central extension of a dihedral, tetrahedral, octahedral or icosahedral group by a cyclic group of even order. Then there is an induced $S^1$-action on $S^3$ so that the action of $\Gamma$ is equivariant, and the orbit space $S^3/\Gamma$ possesses a Seifert fibered structure. Spherical manifolds are divided into five types, $\mathbf{C}$, $\mathbf{D}$, $\mathbf{T}$, $\mathbf{O}$ and $\mathbf{I}$, by the type of $\Gamma$ and are represented by the following normalized Seifert invariants \cite{Seifert:1933}; see also \cite[Chapter 6]{Orlik:1972}.
\begin{itemize}
	\item Type $\mathbf{C}$: $(b;(\alpha_1,\beta_1))$,
	\item Type $\mathbf{D}$: $(b;(2,1),(2,1),(\alpha_3,\beta_3))$,
	\item Type $\mathbf{T}$: $(b;(2,1),(3,\beta_2),(3,\beta_3))$,
	\item Type $\mathbf{O}$: $(b;(2,1),(3,\beta_2),(4,\beta_3))$,
	\item Type $\mathbf{I}$: $(b;(2,1),(3,\beta_2),(5,\beta_3)),$
\end{itemize}
where $\alpha_i>\beta_i>0$ are coprime integers. Up to orientations of the manifolds, we can further assume that $b\geq2$ for $\mathbf{D}$, $\mathbf{T}$, $\mathbf{O}$ and $\mathbf{I}$ type manifolds. In terms of the choices of $\beta_i$'s, we organize the notations of all spherical manifolds, up to the orientations, as follows: 
\begin{itemize}
	\item Type $\mathbf{C}$: $L(p,q)$, $p>q>0, (p,q)=1,$
	\item Type $\mathbf{D}$: $D(p,q)$, $p>q>0, (p,q)=1,$
	\item Type $\mathbf{T}$: $T_{6(b-2)+k}, b\geq2, k=1,3,5,$
	\item Type $\mathbf{O}$: $O_{12(b-2)+k}, b\geq2, k=1,5,7,11,$
	\item Type $\mathbf{I}$: $I_{30(b-2)+k}, b\geq2, k=1,7,11,13,17,19,23,29$.
\end{itemize}
See Table \ref{tab:order} for the identification of these notations with the corresponding Seifert invariants. Notice that the subscript in the notation for $\mathbf{T}$, $\mathbf{O}$ or $\mathbf{I}$ type manifold is equal to the order of $H_1(-;\Z)$ of the manifold divided by $3$, $2$ and $1$ respectively.

\subsection{Spherical 3-manifolds as Dehn surgery manifolds}\label{sec:Dehn-surgery}
Let $S^3_r(K)$ denote the 3-manifold obtained by $r$-framed Dehn surgery of $S^3$ along a knot $K$, and $T_{p,q}$ denote the $(p,q)$-torus knot. Recall that all Dehn surgery manifolds along torus knots are Seifert manifolds.
\begin{lemma}[{\cite[Proposition 3.1]{Moser:1971-1}}, {\cite[Lemma 4.4]{Owens-Strle:2012-1}}]\label{lem:Dehn_surgery}
	Let $p$ and $q$ be relatively prime integers. Then for any rational $r$,
	\begin{equation*}
	S^3_{r}(T_{p,q})\cong -Y(2;(p,q^*),(q,p^*),(m,n))
	\end{equation*}
	where $0<p^*<q$ (resp. $0<q^*<p$) is the multiplicative inverse of $p$ (resp. $q$) modulo $q$ (resp. $p$), and  $\frac{m}{n}=\frac{pq-r}{pq-r-1}$. 
\end{lemma}
In particular, since the Seifert invariant of any spherical 3-manifold of type $\mathbf{T}$, $\mathbf{O}$ or $\mathbf{I}$ can be written as \[\pm Y(2;(2,1),(3,2),(*,*))\] by changing the orientations if necessary, it can be obtained by Dehn surgery along the right-handed trefoil knot $T$ by Lemma \ref{lem:Dehn_surgery}. 
\begin{example}\label{ex:O_11}
	The manifold $O_{12(b-2)+11}$ is homeomorphic to the manifold gotten by $(\frac{2(12b-13)}{4b-5})$-surgery along $T$.
	\begin{align*}
	Y(b;(2,1),(3,1),(4,1))&\cong-Y(-b;(2,-1),(3,-1),(4,-1))\\
	&\cong-Y(-b+3;(2,1),(3,2),(4,3))\\
	&\cong-Y(-b+3+(b-1);(2,1),(3,2),(4,3+4(b-1)))\\
	&\cong S^3_{\frac{2(12b-13)}{4b-5}}(T).
	\end{align*}
\end{example}

\section{Obstructions for 3-manifolds admitting rational ball fillings}\label{sec:obstructions}
In this section, we recall our two main obstructions to 3-manifolds bounding rational homology balls, one from Donaldson's diagonalization theorem and the other from Heegaard Floer correction terms. We also discuss how one can compute the correction terms of spherical manifolds.
\subsection{Donaldson obstruction}\label{sec:Donaldson_obstruction}
One of the main obstructions we shall apply to spherical manifolds bounding rational homology balls or having finite order in $\rhcg$ is given from Donaldson's diagonalization theorem.  

As introduced earlier, if a rational homology 3-sphere $Y$ bounds a negative definite 4-manifold $X$ and a rational homology 4-ball $W$, then the intersection form of $X$, $Q_X$, embeds into the standard negative definite lattice of the same rank, $(\Z^{b_2(X)},\langle-1\rangle^{b_2(X)})$. Namely there exists a map $\rho$ from $(H_2(X;\Z)/\text{Tors}, Q_X)$ to $(\Z^{b_2(X)},\langle-1\rangle^{b_2(X)})$ such that $\rho(v_1)\cdot\rho(v_2)=v_1\cdot v_2$ for any $v_1$ and $v_2$ in $H_2(X;\Z)/\text{Tors}$. We summarize this condition as follows.

\begin{Donaldson obstruction} 
	Let $Y$ be a rational homology 3-sphere. Suppose $Y$ bounds a negative definite smooth 4-manifold $X$. If $Y$ bounds a rational homology ball, then the intersection form of $X$ embeds into the standard negative definite lattice of the same rank.
\end{Donaldson obstruction}

\subsection*{Type-$\mathbf{C}$ manifolds  (lens spaces)}
The $\mathbf{C}$-type spherical manifolds are lens spaces. As mentioned in the introduction, lens spaces bounding rational homology 4-balls are completely classified by Lisca in \cite{Lisca:2007-1}, using the Donaldson obstruction. We briefly recall his strategy and results, which are also applied to some other types of spherical 3-manifolds in Section \ref{sec:Donaldson_2}.

Let $p>q>0$ be relative prime integers, and $L(p,q)$ be the lens space obtained by $(p/q)$-framed Dehn surgery along the unknot. Given the Hirzebruch-Jung continued fraction, $p/q=[a_1,a_2,\dots,a_r]$, let $\Gamma_{p,q}$ be the linear graph with $r$ vertices of weights $-a_1,\dots,-a_r$ consecutively. Note that $-L(p,q)$ bounds a negative definite plumbed 4-manifold corresponding to the graph $\Gamma_{p,q}$. Define a numerical value $\mathcal{I}$ associated to $\Gamma_{p,q}$ as  \[\mathcal{I}(\Gamma)\vcentcolon=\sum_{i=1}^r(|a_i|-3).\]
One of the main ingredients in \cite{Lisca:2007-1} and \cite{Lisca:2007-2} is the complete classification of the direct sums of intersection lattices associated to linear graphs with $\mathcal{I}(-)<0$ that can be embedded into the standard diagonal lattice of the same rank. For the lens spaces whose corresponding definite form are embedded into the standard one with the same rank, Lisca showed that there are in fact rational homology balls bounded by the lens spaces.

\begin{theorem}[\cite{Lisca:2007-1}]
	Suppose $p>q>0$ are relatively prime and $\mathcal{I}(\Gamma_{p,q})<0$. Then
	$L(p,q)$ bounds a rational ball if and only if $\Gamma_{p,q}$ can be embedded into the standard definite lattice of the same rank.	
\end{theorem}
Since lens spaces have a symmetry, $-L(p,q)\cong L(p,p-q)$, and $\mathcal{I}(\Gamma_{p,q})+\mathcal{I}(\Gamma_{p,p-q})=-2$ \cite[Lemma 2.6]{Lisca:2007-1}, either $\Gamma_{p,q}$ or $\Gamma_{p,p-q}$ have negative $\mathcal{I}$ value. Thus the above gives the complete classification for lens spaces admitting rational homology ball fillings.

\subsection{Heegaard Floer correction terms}\label{sec:correction_term_obstruction}

In \cite{Ozsvath-Szabo:2003-2}, Ozsv\'ath and Szab\'o associate to a rational homology 3-sphere $Y$ equipped with a spin$^c$ structure $\mathfrak{t}$ over $Y$, a rational-valued invariant $d(Y,\mathfrak{t})$ called the correction term or $d$-invariant. The invariant is analogous to the $h$-invariant of Fr\o yshov in Seiberg-Witten theory \cite{Froyshov:1996}. The correction terms are rational homology spin$^c$ cobordism invariants and satisfy 
\begin{equation*}
	d(-Y,\mathfrak{t})=-d(Y,\mathfrak{t})
\end{equation*} and \begin{equation*}
	d(Y_1\#Y_2,\mathfrak{t}_1\#\mathfrak{t}_2)=d(Y_1,\mathfrak{t}_1)+d(Y_2,\mathfrak{t}_2).
\end{equation*}
In particular, the correction terms satisfy the following property for rational homology 3-spheres that bound rational homology balls.
\begin{theorem}[{\cite[Proposition 9.9]{Ozsvath-Szabo:2003-2}}]
	If $Y$ is a rational homology $3$-sphere that bounds a rational homology four-ball $W$, then \[d(Y,\mathfrak{t})=0\] for any spin$^c$ structure $\mathfrak{t}$ that extends to $W$.
\end{theorem}
In order to obtain a more effective obstruction, one might need to employ some algebro-topological aspects in this setting. Let $X$ be a closed oriented 3-manifold or a compact oriented 4-manifold. Recall that the set  $\spinc(X)$ of spin$^c$ structures of $X$ is affine isomorphic to $H^2(X;\Z)$. Namely, by fixing a spin$^c$ structure $\mathfrak{s}$ in $\spinc(X)$, we get an isomorphism $H^2(X;\Z)\cong\spinc(X)$, and we denote the image of $\alpha\in H^2(X;\Z)$ in the isomorphism by $\mathfrak{s}+\alpha$. For a 4-manifold $W$ with the boundary 3-manifold $Y$, after fixing a spin$^c$ structure $\mathfrak{s}$ on $W$, we have the following commutative diagram:
\[\begin{CD}
	H^2(W;\Z)    @>\cong>\mathfrak{s}+>  \spinc(W)\\
	@VVV        @VVV\\
	H^2(Y;\Z)     @>\cong>\mathfrak{s}|_Y+>  \spinc(Y),
\end{CD} \]
where the vertical maps are induced by the natural restriction map.

Let $Y$ be a rational homology 3-sphere, and $\lambda$ be the linking form of $Y$. We say a subgroup $\mathcal{M}$ in $H^2(Y;\Z)$ is a \emph{metabolizer} of $Y$ if $\mathcal{M}=\mathcal{M}^\perp$ with respect to $\lambda$. In particular, the order of $\mathcal{M}$ is the square root of that of $H^2(Y;\Z)$. Suppose $W$ is a rational ball bounded by $Y$. Then by the properties of $Y$ and $W$ and the long exact sequence of the pair $(W,Y)$, one can show that the image of the map $H^2(W;\Z)\rightarrow H^2(Y;\Z)$ is a metabolizer in $H^2(Y;\Z)$ \cite{Casson-Gordon:1986-1}. The discussion so far provides the following more detailed obstruction to a rational homology 3-sphere bounding a rational homology ball.
\begin{Correction term obstruction}\label{thm:OS}
	Let $Y$ be a rational homology 3-sphere. If $Y$ bounds a rational homology 4-ball, then there exists a spin$^c$ structure $\mathfrak{s}_0$ on $Y$ and a metabolizer $\mathcal{M}$ in $H^2(Y;\Z)$ such that \[d(Y,\mathfrak{s}_0+\alpha)=0\]
	for any $\alpha\in\mathcal{M}$.
\end{Correction term obstruction}

\subsection{Correction terms of knot surgery manifolds}
Although it is a hard problem in general to compute the correction terms of a rational homology 3-sphere, this becomes particularly tractable if the manifold can be obtained by Dehn-surgery along a knot. Let $K$ be a knot in $S^3$, and let $S^3_{p/q}(K)$ denote the 3-manifold obtained by the $({p}/{q})$-framed Dehn surgery along $K$. Note that there is a natural enumeration $i\in\Z/p\Z$ of spin$^c$ structures over $S^3_{{p}/{q}}(K)$ \cite[Section 4.1]{Ozsvath-Szabo:2003-2}. If $K$ is the unknot $U$, the correction terms of $S^3_{p/q}(U)$ or lens space $L(p,q)$ can be computed by the following recursive formula \cite{Ozsvath-Szabo:2003-2}:
\begin{equation}\label{eq:d_lens_spaces}
d(-L(p,q),i)=\left(\frac{pq-(2i+1-p-q)^2}{4pq}\right)-d(-L(q,r),j)\quad\text{and}\quad d(S^3)=0, 
\end{equation}
where $r$ and $j$ are the reduction modulo $q$ of $p$ and $i$, respectively. According to Jabuka, Robins and Wang \cite{Jabuka-Robins-Wang:2013-1}, the correction terms for lens spaces can be also obtained by a closed formula in terms of the Dedekind-Rademacher sums as
\begin{equation}\label{eq:J-R-W}
	d(L(p,q),i)=2s(q,p;i)+s(q,p)-\frac{1}{2p}
\end{equation}
The Dedekind-Rademacher sum $s(q,p;i)$ is defined as \[s(q,p;i)=\sum_{k=0}^{|p|-1}\bar{B}_1\left(\frac{kq+i}{p}\right)\cdot \bar{B}_1\left(\frac{k}{p}\right)\]and\[s(q,p)=s(q,p;0)-\frac{1}{4},\]
where $\bar{B}_1(x)=x-\lfloor x\rfloor-\frac{1}{2}$ for the usual floor function $x\mapsto\lfloor x\rfloor$.

For a general knot $K$, it is known that the correction terms of $S^3_{p/q}(K)$ can be computed in terms of correction terms of lens spaces and a sequence of non-negative integers $\{V_s(K)\}_{s=0}^\infty$, which are invariants of $K$ derived from the knot Floer chain complex of $K$, by the following formula due to Ni and Wu. 
\begin{theorem}[{\cite[Proposition 1.6]{Ni-Wu:2015}}] 
	Suppose $p,q>0$, and fix $0\leq i\leq p-1$. Then
	\begin{equation}\label{eq:Ni-Wu}
	 d(S^3_{p/q}(K),i)=d(L(p,q),i)-2\max\{V_{\lfloor\frac{i}{q}\rfloor},V_{\lfloor\frac{p+q-1-i}{q}\rfloor}\}.
	\end{equation}
\end{theorem}
If $K$ is a torus knot (more generally, an $L$-space knot), $V_s(K)$ can be gotten by the Alexander polynomial of $K$ \cite{Ozsvath-Szabo:2005-1}. For instance, the right handed trefoil knot $T$ has
\begin{equation*}
V_s(T)=
	\begin{cases}
		1, & s=0 \\
		0, & s>0.
	\end{cases}
\end{equation*}
Thus one may compute correction terms of all spherical manifolds of types $\mathbf{T}$, $\mathbf{O}$ and $\mathbf{I}$, in principle, by the discussion in Section \ref{sec:Dehn-surgery} and the formula above. 

There is another advantage of having Dehn surgery descriptions of our 3-manifolds. Applying the correction term obstruction usually takes more work since we need to determine in advance which spin$^c$ structures over a rational homology 3-sphere can be extended to any hypothetical rational homology 4-ball bounded  by it. See \cite{Owens-Strle:2006-1, Jabuka-Naik:2007-1, Grigsby-Ruberman-Strle:2008-1} for some resolutions of this issue. However, one can easily identify such spin$^c$ structures if the 3-manifold is obtained by Dehn surgery along a knot. 

Recall that if $Y$ is a rational homology 3-sphere bounding a rational homology ball $W$, then the order of $H_1(Y,\Z)$ is $m^2$ for some non-negative integer $m$ from the long exact sequence of the pair $(W,Y)$. In fact, $m$ is the order of the image of the restriction map $H^2(Y;\Z)\rightarrow H^2(W;\Z)$. See\cite[Lemma 3]{Casson-Gordon:1986-1}. Now suppose $Y\cong S^3_{m^2/q}(K)$ for some $m>0$. Then there are exactly $m$ spin$^c$ structures over $Y$ that extend to $W$. On the other hand, it is known that if $Y$ admits at least one correction term of integer value, then there are exactly $m$ spin$^c$ structures on $Y$ whose correction terms are integers, according to Aceto and Golla \cite[Lemma 4.9]{Aceto-Golla:2017-1}. Therefore, we can argue that $Y$ does not bound a rational homology ball, if we find a nonzero integer correction term of a spin$^c$ structure over $Y$.
\begin{proposition}\label{prop:Dehn_surgery}
	Let $Y$ be a manifold obtained by $(m^2/q)$-framed Dehn surgery along a knot in $S^3$. If $Y$ admits a non-vanishing integral correction term for some $\mathfrak{s}\in\text{Spin}^c(Y)$, then $Y$ does not bound a rational homology 4-ball.
\end{proposition}
We call those spin$^c$ structures on $S^3_{m^2/q}(K)$ that admit integer-valued correction terms the \emph{extendable spin$^c$ structures}. In particular, in terms of the natural identification \[\text{Spin}^c(S^3_{m^2/q}(K))\cong \Z/{m^2}\Z,\] the set of extendable spin$^c$ structures is 
\[
	\{[i_0+m\cdot k]\in \Z/{m^2}\Z\mid k=0,\dots,m\},
\]
where 
\[
	i_0=
	\begin{cases}
		\frac{q-1}{2}&\text{for odd }m\text{ and odd }q,\\
		\frac{m+q-1}{2}&\text{for odd }m\text{ and even }q,\\
		\frac{q-1}{2}\text{ or }\frac{m+q-1}{2}&\text{for even }m.
	\end{cases}	
\]
See \cite[Lemma 4.7]{Aceto-Golla:2017-1}. We remark that if $m$ is even, the correction term of either $[\frac{q-1}{2}]$ or $[\frac{m+q-1}{2}]$ has an integer value. In particular, if $m$ is odd, then $i_0$ is the spin$^c$ structure induced from the unique spin structure on $S^3_{m^2/q}(K)$.

\begin{remark}
	During our work, mostly in Section \ref{sec:I_19} and \ref{sec:d_revisited}, we explicitly compute correction terms for some spherical 3-manifolds. For them we use the formula (\ref{eq:Ni-Wu}) after representing our 3-manifolds in terms of Dehn surgery manifold along the trefoil knot. To have the correction terms of lens space in the formula (\ref{eq:Ni-Wu}) explicitly, it is more convenient to use the formula (\ref{eq:J-R-W}) rather than (\ref{eq:d_lens_spaces}) since it can be easily implemented in a mathematical computer program.
\end{remark}

\section{Spherical manifolds of type $\mathbf{D}$ and $\mathbf{T}$}\label{sec:Lecuona}
In this section, we recall Lecuona's results in \cite{Lecuona:2012-1, Lecuona:2017-1}, which allows us to determine the order of manifolds of type $\mathbf{D}$ and of the form $T_{6(b-2)+3}$. We also classify $\mathbf{T}$-type manifolds admitting rational homology ball fillings.

\subsection{Seifert manifolds with complementary legs}
We say two pairs of integers $(\alpha_1,\beta_1)$ and $(\alpha_2,\beta_2)$ in a Seifert invariant are \emph{complementary legs} if $\beta_1/\alpha_1+\beta_2/\alpha_2=1$. In \cite{Lecuona:2012-1, Lecuona:2017-1}, Lecuona studied the set of 3-legged Seifert manifolds with complementary legs. We recall Lecuona's results.
\begin{proposition}[{\cite[Proposition 3.1]{Lecuona:2017-1}}, {\cite[Section 3.1]{Lecuona:2012-1}}]\label{prop:complementary_legs} 
	Let $Y$ be a Seifert manifolds with an invariant, \[(b;(\alpha_1, \beta_1),(\alpha_2, \beta_2),(\alpha_3, \beta_3)).\] Suppose $(\alpha_1, \beta_1)$ and $(\alpha_2, \beta_2)$ are complementary legs, namely  $\beta_1/\alpha_1+\beta_2/\alpha_2=1$. Then $Y$ is rational homology cobordant to the manifold with the invariant $(b-1;(\alpha_3, \beta_3))$.
\end{proposition}
One can find a more general statement for Seifert manifolds with more than $3$-legs and a complementary legs in \cite[Lemma 6.2]{Aceto-Golla:2017-1}. Lecuona further showed that if such a Seifert manifold bounds a rational homology ball, then the corresponding Montesinos links admit a ribbon surface. More precisely,
\begin{proposition}[{\cite[Proposition 3.4]{Lecuona:2017-1}}]\label{prop:Lecuona_ribbon}
	Let $\Gamma$ be a 3-legged star shaped graph with two complementary legs. The Montesinos link $ML_
	\Gamma\subset S^3$ associated to $\Gamma$ is the boundary of a ribbon surface $F$ with $\chi(F)=1$ if and only if the Seifert space $Y_\Gamma$ associated to $
	\Gamma$ is the boundary of a rational homology ball.
\end{proposition}

\subsection{Type-D manifolds (Prism manifolds)}\label{sec:type_D}
Recall that a type-$\mathbf{D}$ manifold admits the Seifert invariant \[(b_0; (2, 1), (2, 1), (\alpha_3,\beta_3)),\] such that $b_0\geq2$ and $\alpha_3>\beta_3>0$ are coprime integers, up to the orientations. The manifolds are usually enumerated by two coprime integers $p>q>0$, like for lens spaces. Let $\frac{p}{q}=[b_0,b_1,\dots,b_r]$, and $D(p,q)$ denote the manifold homeomorphic to the boundary of the 4-manifold corresponding to the left plumbing graph in Figure \ref{fig:manifold_D}. Since $(2, 1)$ and $(2, 1)$ are complementary pairs, we have the following direct corollary of Proposition \ref{prop:complementary_legs}.

\begin{figure}
	\centering
	\includegraphics[width=0.9\textwidth]{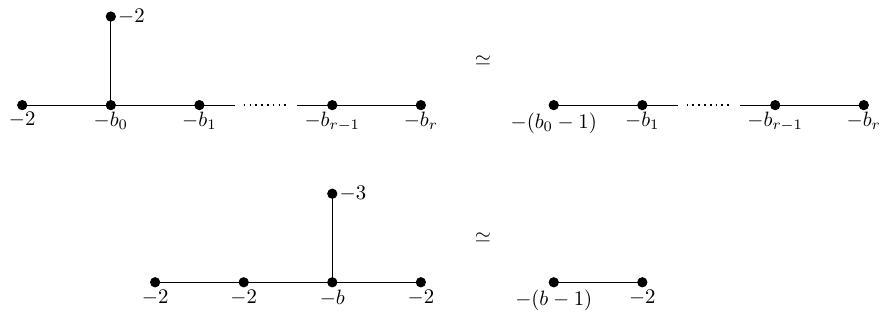}
	\caption{The canonical plumbing graphs of $D(p,q)$ and $T_ {6(b-2)+3}$, and linear graphs associated to manifolds that are rational homology cobordant to them respectively.}
	\label{fig:manifold_D}
\end{figure}

\begin{proposition}
	Let $p>q>0$ be coprime integers. The dihedral manifold $D(p,q)$ is rational homology cobordant to the lens space $-L(p-q,q)$.
\end{proposition}
In particular,  the manifold $D(p,q)$ has the same order in $\rhcg$ as that of the lens space $L(p-q,q)$. Notice that if $b_0=2$, then $\frac{p-q}{q}<1$. Since $L(p-q,q)\cong L(p-q,q')$ for $q'\equiv q$ modulo $p-q$, we have that $D(p,q)$ bounds a rational homology ball if and only if $\frac{p-q}{q'}\in \mathcal{R}$, where $0<q'<p-q$ is the reduction of $q$ modulo $p-q$, as the statement in Theorem \ref{thm:main}. 

\subsection{Type-T}\label{sec:type_T}
The $\mathbf{T}$-type manifolds of the form $T_ {6(b-2)+3}$ admitting the Seifert invariant 
	\[(b;(2,1),(3,1),(3,2)),\] 
also possess complementary legs, $(3,1)$ and $(3,2)$. Thus the manifold  $T_{6(b-2)+3}$ is rational homology cobordant to the lens space $-L(2b-3,2)$ by Proposition \ref{prop:complementary_legs}. Then by Lisca's result on the order of lens spaces in $\rhcg$ \cite[Corollary 1.3]{Lisca:2007-2}, the order of $T_ {6(b-2)+3}$ is given as
\begin{itemize}
	\item $1$ if $b=2,6$,
	\item $2$ if $b=4$, and
	\item $\infty$ otherwise.
\end{itemize}

In fact, this together with a homological condition for rational homology spheres bounding rational homology balls gives a complete answer for type-$\mathbf{T}$ manifolds admitting rational ball filling.

Let $Y$ be a $\mathbf{T}$-type spherical manifold, which has the normalized Seifert form \[(b;(2,1),(3,\beta_2),(3,\beta_3)).\]  Observe that the order of $H_1(Y;\Z)$ is equal to 
\begin{equation*}
3\left|6b-3-2(\beta_2+\beta_3)\right|.
\end{equation*}
If $Y$ bounds a rational homology 4-ball, then $|H_1(Y;\Z)|$ is a perfect square. This forces $Y$ to have $\{\beta_2,\beta_3\}=\{1, 2\}$, namely $Y$ is of the form $T_ {6(b-2)+3}$. Therefore, any manifolds of the form $T_{6(b-2)+1}$ or $T_{6(b-2)+5}$ cannot bound a rational homology 4-ball.
\begin{proposition} 
	A spherical manifold $Y$ of type $\mathbf{T}$ bounds a smooth rational homology 4-ball if and only if  $Y$ or $-Y$ is homeomorphic to $T_3$ or $T_{27}$.
\end{proposition}  
	
\section{$\mathbf{O}$- and $\mathbf{I}$-type manifolds bounding rational homology balls}\label{sec:rational_ball}
In this section, we classify the spherical 3-manifolds of type $\mathbf{O}$ and $\mathbf{I}$ bounding rational homology 4-balls. In fact, this can be answered as a corollary of the results in Section \ref{sec:order}, where we will determine the order of those spherical manifolds in $\rhcg$. Nonetheless we present this section because the argument becomes much easier if one just wants to determine whether a spherical manifold bounds a rational homology ball or not, instead of the exact order of the manifold. This is because the condition that the order of $H_1$ of a rational homology 3-sphere bounding a rational homology ball is a square number reduces the cases we need to examine. Whereas this condition is not applicable to the connected sums of a manifold since any even number of connected sum admits a square order of $H_1$.  

\subsection{Type-O}
The order of $H_1$ of an $\mathbf{O}$-type manifold, which has a Seifert invariant $$(b;(2,1),(3,\beta_2),(4,\beta_3)),$$ where $b\geq2$, and $\beta_2\in\{1,2\}$ and $\beta_3\in\{1,3\}$, equals
$$2\left|12b-6-4\beta_2-3\beta_3\right|.$$ Note that it cannot be a square number since $\beta_3$ is coprime to $4$. Therefore, an $\mathbf{O}$-type manifold cannot bound any rational homology 4-ball.

\subsection{Type-I}
Let $Y$ be an $\mathbf{I}$-type manifold, which has a Seifert invariant 
\[(b;(2,1),(3,\beta_2),(5,\beta_3)),\]
where $b\geq2$, $\beta_2\in\{1,2\}$ and $\beta_3\in\{1,2,3,4\}$. Note that the order of $H_1(Y;\Z)$ equals \[\left|30(b-2)+45-10\beta_2-6\beta_3\right|.\] Considering the quadratic residues modulo $30$, the only cases to make $|H_1(Y,\Z)|$ a square number are $(\beta_2,\beta_3)=(2,4)$ or $(\beta_2,\beta_3)=(2,1)$. Hence if an $\mathbf{I}$-type manifold bounds a rational homology ball, then it should be one of the following families:
\[I_{30(b-2)+1}=Y(b;(2,1),(3,2),(5,4))\] or \[I_{30(b-2)+19}=Y(b;(2,1),(3,2),(5,1)).\]

\subsubsection{Manifolds of the form $I_{30(b-2)+1}$}\label{sec:I_1} 
We can show that the manifold $I_{30(b-2)+1}$ cannot bound any rational homology ball by using the Donaldson obstruction. Note that the manifold $I_{30(b-2)+1}$ bounds the plumbed definite 4-manifold $X$ in Figure \ref{fig:manifold_I_1}. Let us label the vertices with weight $-2$ on the plumbing diagram by $v_1$ to $v_7$ as depicted in Figure \ref{fig:manifold_I_1}. We let $v_i$ simultaneously denote a generator of $H_2(X;\Z)$ represented by the sphere corresponding to the vertex. Suppose there is an embedding $\rho$ of $Q_X$ into the standard negative definite lattice $( \Z^n,\langle-1\rangle^n )$ with the standard basis $\{e_1,\dots,e_n\}$; namely $e_i\cdot e_i=-1$ for each $i$ and $e_i\cdot e_j=0$ for $i\neq j$. Since $v_1\cdot v_1=-2$, we have $\rho(v_1)=\pm e_i\pm e_j$ for some $i$ and $j$. After re-indexing and re-scaling by $\pm1$, we may assume that $\rho(v_1)=e_1-e_2$ without loss of generality. By a similar argument, the image of $v_2$ has the form $\rho(v_2)=e_2-e_3$ since $v_1\cdot v_2=1$. For the image of $v_3$ on the embedding, there could be two choices: $\rho(v_3)=e_3-e_4$ or $-e_2-e_1$. However, assuming the latter implies that $v_1\cdot v_4\equiv v_3\cdot v_4\equiv 1 \ (\text{mod}\ 2)$, which gives a contradiction. Hence we admit the former case. By performing this procedure consecutively, the embedding is expressed as $\rho(v_1)=e_1-e_2$, $\rho(v_2)=e_2-e_3$, $\rho(v_3)=e_3-e_4$, $\rho(v_4)=e_4-e_5$, $\rho(v_5)=e_6-e_7$, $\rho(v_6)=e_7-e_8$ and $\rho(v_7)=e_9-e_{10}$, up to the automorphisms of $(\Z^n,\langle-1\rangle^n)$. Therefore, the rank of the image of $\rho$ is at least $10$, and $I_{30(b-2)+1}$ cannot bound any rational homology ball by the Donaldson obstruction. 
\begin{figure}[t!]
	\includegraphics[width=0.6\textwidth]{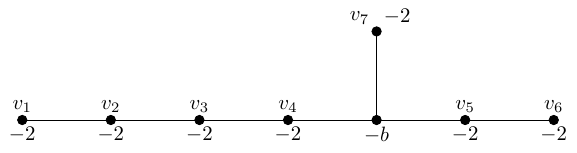}
	\caption{The canonical plumbing graph of the manifold $I_{30(b-2)+1}$.}
	\label{fig:manifold_I_1}
\end{figure}

\subsubsection{Manifolds of the form $I_{30(b-2)+19}$}\label{sec:I_19} 
Suppose $I_{30(b-2)+19}$ bounds a rational homology ball. Then \[|H_1(I_{30(b-2)+19};\Z)|=30(b-2)+19\] must be a square number $m^2$. In order to satisfy this property, $m^2$ should be equal to $(30k+13)^2$ or $(30k+23)^2$ for some $k\in\Z$ by considering square numbers in $\Z/30\Z$. Let us further assume this. 

For the manifold $I_{30(b-2)+19}$, we cannot apply the same obstruction used for $I_{30(b-2)+1}$ since the intersection form of the canonical plumbed manifold of $I_{30(b-2)+19}$ can be embedded into the standard definite diagonal lattice of the same rank. We first observe this. Let $v_1,\dots,v_5$ denote the basis vectors of the intersection form $Q_X$ of the canonical plumbed 4-manifold $X$ corresponding to $I_{30(b-2)+19}$ as depicted in Figure \ref{fig:I_19}, and let $\{e_1,\dots,e_5\}$ be the standard basis of the diagonal lattice $(\Z^5,\langle-1\rangle^5)$. If $m^2=(30k+13)^2$ (i.e. $b=30k^2+26k+7$), we have an embedding of $Q_X$ into $\langle-1\rangle^5$ so that
\begin{align*}
	\rho(v_1)&=e_1-e_2,\\
	\rho(v_2)&=e_3-e_4,\\
	\rho(v_3)&=e_4-e_5,\\
	\rho(v_4)&=e_1+e_2-(e_3+e_4+e_5),\text{ and}\\
	\rho(v_5)&=e_1+2e_2+e_3+e_4+2e_5-(k+1)\{3(e_1+e_2)+2(e_3+e_4+e_5)\}.
\end{align*}
If $m^2=(30k+23)^2$ (i.e. $b=30k^2+46k+19$), then we get an embedding by assigning $v_5$ instead as
\begin{align*}
\rho(v_5)&=e_2-e_3-e_4-(k+1)\{3(e_1+e_2)-2(e_3+e_4+e_5)\}.
\end{align*} 
\begin{figure}
	\includegraphics{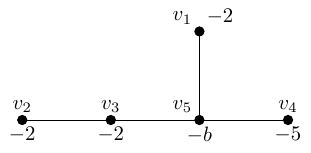}
	\caption{The canonical plumbing graph of the manifold $I_{30(b-2)+19}$.}
	\label{fig:I_19}
\end{figure}

Instead of the Donaldson obstruction, we make use of the condition from Heegaard Floer correction terms. Observe that $I_{30(b-2)+19}$ can be obtained by $(-\frac{30(b-2)+19}{5b-6})$-framed Dehn-surgery along the left-handed trefoil knot. We first consider the correction terms for the associated lens spaces. 
\begin{lemma}\label{lem:d_invariant}
	Let $k\in\Z$ and $m^2=(30k+13)^2$ and $q=150k^2+130k+29$, or $m^2=(30k+23)^2$ and $q=150k^2+230k+89$, and $i_0=\frac{q-1}{2}$. Then 
	$$d(L(m^2,q),i_0+2|m|)=6$$
	for $k\neq0,-1$.
\end{lemma}
\begin{proof}
	This is obtained by a direct computation using the reciprocal formula (\ref{eq:d_lens_spaces}) for the correction terms of lens spaces. We give tables of triple $(p,q,i)$'s in the computation.  The condition $k\neq0,-1$ ensures that $p>q>0 $ and $p>i\geq 0$ in each steps.
	
	{\centering
	\begin{table*}[h!]
		{\renewcommand{\arraystretch}{1.2}
		\begin{tabular}{c c c c}
		\hline
		Steps & $p$ & $q$ & $i$ \\ 
		\hline 
		$1$ & $(30k+13)^2$ & $150k^2+130k+29$ & $75k^2+65k+14+2m $ \\
		$2$ & $150k^2+130k+29$ & $150k^2+130k+24$ & $75k^2+65k+14+2m $ \\
		$3$ & $150k^2+130k+24$ & $5$ &$75k^2+65k+14+2m $ \\
		$4$ & 5 & 4 & 0  \\
		\hline 
		\end{tabular}}
	\end{table*}
	}

	{\centering
	\begin{table}[h!]
	{\renewcommand{\arraystretch}{1.2}
	\begin{tabular}{c c c c}
	
		\hline
		Steps & $p$ & $q$ & $i$ \\ 
		\hline
		$1$ & $(30k+23)^2$ & $150k^2+230k+89$ & $75k^2+115k+44+2m $ \\
		$2$ & $150k^2+230k+89$ & $150k^2+230k+84$ & $75k^2+115k+44+2m $ \\
		$3$ & $150k^2+230k+84$ & $5$ &$75k^2+115k+44+2m $ \\
		$4$ & 5 & 4 & 0  \\
		\hline 
	\end{tabular}}
\end{table}
}
\end{proof}

Let $T$ be the right-handed trefoil knot. Then by the formula (\ref{eq:Ni-Wu}) we have
\begin{align*}
	d(I_{30(b-2)+19}, i_0+2|m|)&=d(S^3_{-\frac{30(b-2)+19}{5b-6}}(-T),i_0+2|m|)\\
	&=-d(S^3_{\frac{30(b-2)+19}{5b-6}}(T),i_0+2|m|)\\
	&=-(d(L(m^2,q),i_0+2|m|)-2\max\{V_{\lfloor\frac{i}{q}\rfloor}(T),V_{\lfloor\frac{p+q-1-i}{q}\rfloor}(T)\}).
\end{align*}
Therefore, $I_{30(b-2)+19}$ cannot bound a rational homology ball for any $b\neq3,7,11,19$ by Lemma \ref{lem:d_invariant} and Proposition \ref{prop:Dehn_surgery}. Similarly we can compute that 
\[
	d(L(23^2,89), i_0+2\cdot23)=4,
\] 
and hence $I_{529}$ ($b=19$) cannot bound a rational ball neither. The manifolds $I_{289}$ ($b=11$) and $I_{169}$ ($b=7$) also admit nonvanishing integral correction terms,
\[
	d(I_{289}, i_0+2\cdot 17)=2\quad\text{and}\quad d(I_{169}, i_0+6\cdot 13)=2,
\]
by the straightforward computation using the formula (\ref{eq:Ni-Wu}) and the correction term formula of \cite{Jabuka-Robins-Wang:2013-1} for lens spaces.
\begin{proposition}
	The spherical manifold of the form $I_{30(b-2)+19}$, $b\geq2$, does not bound a rational homology ball if $b\neq3$.
\end{proposition}
\subsection{The manifold $I_{49}$ bounding a rational homology ball} The manifold $I_{49}$ is the only remaining case of spherical manifolds that we have not determined if it bounds a rational homology ball. Note that all correction terms of extendable spin$^c$ structures on $I_{49}$ vanish, i.e.
\[
	d(I_{49},i_0+7s)=0, 
\]
for $i_0=4$ and $s=0,\dots,6$. We shall show that the manifold $I_{49}$ in fact bounds a rational homology ball by proving a stronger statement that this manifold is the double cover of $S^3$ branched along a \emph{ribbon} knot. Recall that any Seifert fibered rational homology sphere is the double cover of $S^3$ along a Montesinos link \cite{Montesinos:1973-1} (See also \cite[Chapter 2]{Lecuona:2012-1}). The corresponding link to a Seifert manifold can be found from the plumbing graph associated to the manifold as follows . To each vertex with weight $n$ in the graph, we associate $D^1$ bundle over $S^1$ with $n$ half twists, and to each adjacent vertices we perform plumbings of the corresponding pair of bundles. Then the boundary of the 2-manifold constructed in the 3-sphere is the link to produce the Seifert manifold by the double branched cover. We denote the Montesinos link corresponding to the Seifert invariant $(b;(\alpha_1,\beta_1),\dots,(\alpha_r,\beta_r))$ by 
$M(b;(\alpha_1,\beta_1),\dots,(\alpha_r,\beta_r))$. For instance, the Montesinos knot associated to the manifold $I_{49}$ is depicted below in Figure \ref{fig:I_49}.
\begin{proposition}\label{prop:ribbon}
	The Montesinos knot $M(3;(2,1),(3,2),(5,1))$ is  a ribbon knot.
\end{proposition}
\begin{figure}[t]
	\includegraphics[width=0.8\textwidth]{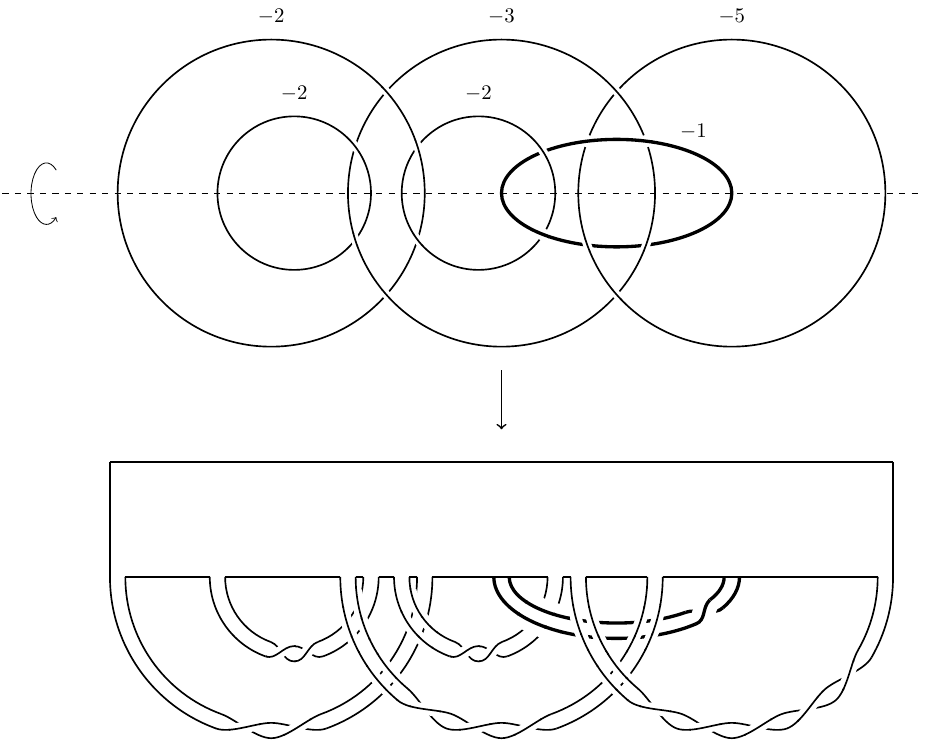}
	\caption{A surgery diagram of $I_{49}$ as a strongly invertible link and its natural branch set, the Montesinos knot $M(3;(2,1),(3,2),(5,1))$. The band associated to the thick $(-1)$-framed unknot induces a ribbon move to the two component unlink.}
	\label{fig:I_49}
\end{figure} 
\begin{proof}
	It is a well-known fact that if there is a band that induces a ribbon move from a knot to the two component unlink, then the knot is ribbon. A band attached to the Montesinos knot $M(3;(2,1),(5,1),(3,2))$ is given in Figure \ref{fig:I_49}, and one can check that this can be isotoped to the two component unlink by performing the ribbon move along the band.
	
	We found the band by using a technique of Lecuona in \cite[Section 3]{Lecuona:2012-1}. Lecuona considered the class of Montesinos knots whose branched cover has the plumbing graph with $\mathcal{I}(-)\leq-2$, and gave an algorithm to find a band if it is a ribbon knot. Although the graph $\Gamma$ associated to our Montesinos knot $M(3;(2,1),(5,1),(3,2))$ has $\mathcal{I}(\Gamma)=-1$, Lecuona's technique was still useful to find a band attached to it. Considering an embedding of the intersection lattice of $I_{49}$ to the standard diagonal one of the same rank, we find a $(-1)$-framed unknot added to the surgery diagram of $I_{49}$, as a \emph{strongly invertible link} depicted in Figure \ref{fig:I_49}. One can easily check that consecutive blow-down's from the framed link diagram with the $(-1)$-framed unknot results in $S^1\times S^2$. In this case, the band corresponding to the branch set of the $(-1)$-framed unknot, induces a ribbon move to the two component unlink. 
\end{proof}

The discussion in the last two sections gave the proof of Theorem \ref{thm:main}, responding which spherical 3-manifolds bound rational homology balls. 

\section{Orders of spherical 3-manifolds in $\Theta^3_\Q$}\label{sec:order}
In this section, we discuss more generally the order of spherical 3-manifolds in $\Theta^3_\Q$. The order of lens spaces were completely determined by Lisca in \cite{Lisca:2007-2}. In Section \ref{sec:Lecuona}, we observed that any $\mathbf{D}$-type manifolds and manifolds of the form $T_{6(b-2)+3}$ are rational homology cobordant to lens spaces, and their order is the same as that of the corresponding lens spaces. Now we determine the order of all other types of spherical manifolds, using both Donaldson and correction term obstructions again.

\subsection{Finite order obstruction from Heegaard Floer correction terms}\label{sec:mu_bar}	
The Heegaard Floer correction term for a certain spin$^c$ structure can be used to obstruct some rational homology 3-spheres to have finite order in $\Theta^3_\Q$. Note that if a rational homology 3-sphere $Y$ has odd $|H_1(Y;\Z)|$, then it admits a unique spin structure on $Y$ since $H_1(Y;\Z/2\Z)=0$.
\begin{proposition}\label{prop:infinite}
	Let $Y$ be a rational homology 3-sphere for which $|H_1(Y;\Z)|$ is odd, and let $\mathfrak{s}_0$ be the spin$^c$ structure induced from the unique spin structure on $Y$. If $Y$ has finite order in $\Theta^3_\Q$, then \[d(Y,\mathfrak{s}_0)=0.\]
\end{proposition}
\begin{proof}
	Let $Y$ be a rational homology 3-sphere with odd $|H_1(Y;\Z)|$. We first claim that the spin$^c$ structure $\mathfrak{s}_0$ induced from the unique spin structure on $Y$ extends to any rational homology ball bounded by $Y$. Recall that if a spin$^c$ structure on $Y$ extends to $W$, then its conjugate also does. Since the set of extendable spin$^c$ structures on $Y$ to $W$ is also of odd order (the square root of $|H_1(Y;\Z)|$), it must contain $\mathfrak{s}_0$, which is the unique spin$^c$ structure preserved by the conjugation, i.e. the one from the spin structure. We remark that the same argument appears in the proof of \cite[Proposition 4.2]{Stipsicz:2008-1}, although it is stated only for rational homology 3-spheres obtained from plumbed graphs. 

	Notice that $\#^n\mathfrak{s}_0$ is the spin$^c$ structure induced from the unique spin structure on $\#^nY$. Hence if $\#^nY$ bounds a rational homology ball for some $n$, in other words, $Y$ has finite order in $\Theta^3_\Q$, then
	\begin{equation*}
		d(\#^nY,\#^n\mathfrak{s}_0)=n\cdot d(Y,\mathfrak{s}_0)=0
	\end{equation*}
	by the property of correction terms under the connected sum and Theorem \ref{thm:OS}.
\end{proof}
\begin{remark}
	We remark that the parity condition in the proposition is necessary. For instance, it is well known that $L(p^2,p-1)$ bounds a rational homology ball, but one can check that the correction terms for spin structures on $L(p^2,p-1)$ are non-vanishing when $p$ is even.
\end{remark}
Now we apply the above to spherical 3-manifolds $Y$ with odd $|H_1(Y;\Z)|$. In \cite{Stipsicz:2008-1} Stipsicz showed that a correction term for spin spherical 3-manifold, or more generally for any spin link of a rational surface singularity, can be identified with the $\overline{\mu}$-invariant of Neumann \cite{Neumann:1980-1} and Siebenmann \cite{Siebenmann:1980-1} as
\begin{equation*}
	-4d(Y,\mathfrak{s}_0)=\overline{\mu}(Y,\mathfrak{s}_0).
\end{equation*}
In particular, $\overline{\mu}(Y,\mathfrak{s}_0)$ can be computed from the plumbing graph of the canonical definite 4-manifold $X$ of $Y$ \cite{Neumann:1980-1} (See also Section 2 of \cite{Stipsicz:2008-1}). Let $\Sigma$ denote the \emph{Wu surface} of $\mathfrak{s}_0$, namely the surface that represents the Poincar{\'e} dual of $c_1(\mathfrak{s}_0)$ and has coordinates $0$ or $1$ in the basis of $H_2(X;\Z)$ represented by embedded 2-spheres in the plumbing diagram of $X$. Then \[\overline{\mu}(Y,\mathfrak{s}_0)={\sigma(X)-[\Sigma]^2},\]
where $\sigma(X)$ is the signature of the intersection form of $X$.
\begin{corollary}
	The following spherical manifolds admit non-vanishing correction terms for the spin structures, and hence have infinite order in $\Theta^3_\Q$.
	\begin{itemize}
		\item $T_{6(b-2)+k}$ with any $b$ and $k=1,5$.
		\item $I_{30(b-2)+k}$ with any $b$ and $k=1,7,11,13,29$.
		\item $I_{30(b-2)+17}$ with odd $b$.
		\item $I_{30(b-2)+19}$ and $I_{30(b-2)+23}$ with even $b$.
	\end{itemize}
\end{corollary}
\begin{proof}
 	For those manifolds, we compute the $\overline\mu$-invariants (and hence $d$-invariants) for the spin structures by the algorithm in \cite[Section 2]{Stipsicz:2008-1} using the canonical plumbing graphs. For instance, the plumbing graph associated to $T_{6(b-2)+1}$ is depicted in Figure \ref{fig:T-1}. The Wu surface $\Sigma$ corresponding to the spin structure on $T_{6(b-2)+1}$ is the sphere corresponding to the $-2$ vertex on the short leg if $b$ is odd, and the empty surface if $b$ is even. Hence, \begin{equation*}
		-4d(T_{6(b-2)+1},\mathfrak{s}_0)=\overline\mu(T_{6(b-2)+1},\mathfrak{s}_0)=\sigma(X)-[\Sigma]^2=
		\begin{cases}
			-4\quad\text{for odd }b\\
			-6\quad\text{for even }b.\end{cases}
	\end{equation*}
	By applying the algorithm to other types of manifolds listed, one can check that the correction terms for the spin structures on the manifolds do not vanish. Therefore, those have infinite order in $\rhcg$ by Proposition \ref{prop:infinite}.
	\begin{figure}[t!]
		\includegraphics{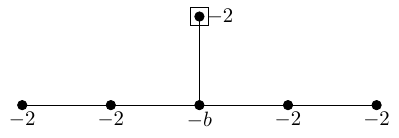}
		\caption{The canonical plumbed graph of $T_{6(b-2)+1}$, and the boxed vertex representing the Wu surface corresponding to the spin structure on it for even $b$.}
		\label{fig:T-1}
	\end{figure}
\end{proof}
On the other hand, the manifolds of type $I_{30(b-2)+17}$ with even $b$, and  $I_{30(b-2)+19}$ and $I_{30(b-2)+23}$ with odd $b$ have vanishing correction terms for spin structures, and any $\mathbf{O}$ type manifolds admit even order first homology group. Hence we cannot apply Proposition \ref{prop:infinite} to them.

\subsection{Finite order obstruction from Donaldson's theorem}
\label{sec:Donaldson_infinite}
The Donaldson obstruction can be also used to give information for the order of rational homology 3-manifolds. If a rational homology 3-sphere $Y$ bounds a negative definite 4-manifold $X$ and has order $n$ in $\rhcg$ then the direct sum of $n$ copies of the intersection form of $X$ is embedded into the standard diagonal lattice of rank $n\cdot b_2(X)$. This condition was sufficient to obstruct sums of lens spaces to have finite order by Lisca \cite{Lisca:2007-2}. For the non-cyclic spherical manifolds, we consider two kinds of definite fillings of them. One is the canonical definite plumbed 4-manifold we have used in Section \ref{sec:I_1}, and the other is the 4-manifold induced from the Dehn surgery descriptions of them, which we will introduce first.

\subsubsection{Donaldson obstruction using Dehn surgery descriptions of spherical manifolds}\label{sec:Donaldson_2}
Let $p>q>0$ be relatively prime integers and $K$ be a knot in $S^3$. Let $\frac{p}{q}$ admit the Hirzebruch-Jung continued fraction $[a_1,a_2,\dots,a_r]$, with $a_i\geq2$, and let $\Gamma_{p,q}$ be the linear graph with weights $a_1,\dots,a_r$ consecutively. By Rolfsen's twist $S^3_{p/q}(K)$ bounds a definite 4-manifold $X$ with intersection form isomorphic to $Q_{\Gamma_{p,q}}$, the incidence form of $\Gamma_{p,q}$; see Figure \ref{fig:Dehn_surgery_filling} for a framed link diagram of $X$. As we introduced in Section \ref{sec:Donaldson_obstruction}, Lisca studied the embedding of sums of intersection lattices $Q_{\Gamma_{p,q}}$ into the standard diagonal lattice of the same rank, provided that $\mathcal{I}(\Gamma_{p,q})<0$. Thus the result of Lisca can be also applied to $S^3_{p/q}(K)$ such that $\mathcal{I}(\Gamma_{p,q})<0$. 

\begin{figure}[tb!]
	\includegraphics[width=0.8\textwidth]{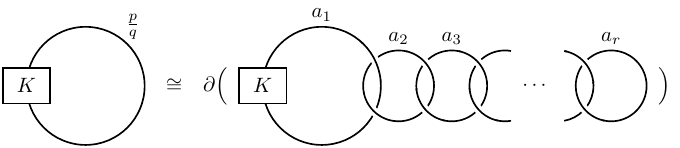}
	\caption{A definite 4-manifold bounded by a knot surgery manifold.}
	\label{fig:Dehn_surgery_filling}
\end{figure}

\begin{proposition}\label{prop:Dehn-surgery}
	Let $p>q>0$ be relatively prime integers such that $\mathcal{I}(\Gamma_{p,q})<0$. Then for any knot $K$ in $S^3$, the order of $S^3_{p/q}(K)$ in $\rhcg$ is greater than or equals to that of $L(p,q)$. 
\end{proposition}
\begin{proof}
	In Lisca's works \cite{Lisca:2007-1,Lisca:2007-2}, a lower bound for the order of a lens spaces $L(p,q)$ in $\rhcg$ is obtained by the Donaldson obstruction using the definite lattice $\oplus^nQ_{\Gamma_{p,q}}$ bounded by the lens spaces. Therefore, the lower bound is also applicable to the 3-manifolds that bound a 4-manifold with the same intersection lattice $Q_{\Gamma_{p,q}}$. Moreover, for the lens spaces, the bound was turned out to be the exact order in the Lisca's work. 
\end{proof}

For example, consider manifolds of the form $O_{12(b-2)+11}$. By Lemma \ref{lem:Dehn_surgery}, we have the following surgery description of it (see Example \ref{ex:O_11}):
\begin{equation*}
O_{12(b-2)+11}\cong S^3_{\frac{2(12b-13)}{4b-5}}(T).
\end{equation*}
The surgery coefficients have the following continued fraction:
\begin{equation*}
\frac{2(12b-13)}{4b-5}=
\begin{cases}
[8,2,2]&\text{ if } b=2\\
[7,\underbrace{2,\dots,2}_{b-3},3,2,2] &\text{ if } b\geq3.\\
\end{cases}
\end{equation*}
Notice that if $b>5$, then the corresponding linear graph has $\mathcal{I}(\Gamma_{2(12b-13), 4b-5})<0$. 

Similarly, we have
\begin{equation*}
I_{30(b-2)+17}\cong S^3_{\frac{30b-43}{5b-8}}(T) \quad\text{and}\quad I_{30(b-2)+23}\cong S^3_{\frac{30b-37}{5b-7}}(T),
\end{equation*}
and the continued fraction of the surgery coefficients are given as
\begin{equation*}
	\frac{30b-43}{5b-8}=
	\begin{cases}
		[9,2]&\text{ if } b=2\\
		[7,\underbrace{2,\dots,2}_{b-3},4,2] &\text{ if } b\geq3\\
	\end{cases}
\end{equation*}
and
\begin{equation*}
	\frac{30b-37}{5b-7}=
	\begin{cases}
		[8,3]&\text{ if } b=2\\
		[7,\underbrace{2,\dots,2}_{b-3},3,3] &\text{ if } b\geq3.\\
	\end{cases}
\end{equation*}
Hence if $b>7$ in both cases, each linear graph corresponding to the surgery coefficients have $\mathcal{I}(-)<0$.

\begin{corollary}
The following spherical manifolds have infinite order in $\rhcg$.
\begin{itemize}
	\item $O_{12(b-2)+11}$ for any $b>5$
	\item $I_{30(b-2)+17}$ for any $b>7$
	\item $I_{30(b-2)+23}$ for any $b>7$
\end{itemize}	
\end{corollary}
\begin{proof}
	This directly follows from the result of \cite{Lisca:2007-1, Lisca:2007-2} and Proposition \ref{prop:Dehn-surgery}. More precisely one check that the incidence forms of linear plumbing graphs corresponding to the surgery slopes of the manifolds listed above are not included in the Lisca's set of linear lattices or sums of them that can be embedded into the standard diagonal one with the same rank. 
\end{proof}
We remark that the remaining spherical manifolds that we have not determined the order, admit the Dehn-surgery descriptions (by Lemma \ref{lem:Dehn_surgery}) whose corresponding linear graphs have positive $\mathcal{I}$ values. 

\subsubsection{Donaldson obstruction using canonical plumbed 4-manifolds}\label{sec:Donaldson_1}
For the manifolds of type $O_{12(b-2)+1}$, $O_{12(b-2)+5}$, and $O_{12(b-2)+7}$, we make use of their canonical plumbed 4-manifolds to apply the Donaldson obstruction and show they have infinite order in $\Theta^3_\Q$. 

\begin{proposition}\label{prop:infinite_order_Donaldson}
	Let $X$ be the canonical negative definite 4-manifold with the boundary $O_{12(b-2)+k}$ for $b\geq2$ and $k=1,5\text{ or }7$. Then the direct sum of $n$-copies of $Q_X$ does not embed into the standard diagonal lattice of rank $n\cdot \text{rk}(Q_X)$ for any $n\geq1$.
\end{proposition}

\begin{proof}
		We first consider the manifold of the form $O_{12(b-2)+1}$. Let $Q_X$ be the intersection form of the canonical negative definite 4-manifold $X$ with the boundary $O_{12(b-2)+1}$, and let $Q_X^n$ denote the direct sum of $n$ copies of $Q_X$. Suppose that there is an embedding $\rho$ of $Q_X^n$ into the standard lattice $\langle-1\rangle^r$ for some $r>0$. Label the basis vectors of the $i$-th copy of $Q_X^n$ by $v_1^i,\dots,v_7^i$, corresponding to the spheres in the plumbing diagram of $X$ as seen in the top of Figure \ref{fig:Octa}. Let $\{e_1,\cdots,e_r\}$ be the standard basis of $\langle-1\rangle^r$. Then after re-indexing and re-scaling by $\pm 1$, we may assume $\rho(v_1^1)=e_1-e_2$, $\rho(v_2^1)=e_3-e_4$, $\rho(v_3^1)=e_4-e_5$, $\rho(v_4^1)=e_5-e_6$, $\rho(v_6^1)=e_7-e_8$ and $\rho(v_7^1)=e_8-e_9$ for the embedding of the first copy of $Q_X$ as usual. Then there are essentially two choices of the image of $v_1^2$, namely $\rho(v_1^2)=\pm(e_1+e_2)$ or $e_{10}-e_{11}$. However notice that the former case cannot happen since $1=v_1^1\cdot v_5^1\equiv v_1^2\cdot v_5^1$ (mod $2$) assuming this. By using this argument inductively, we may assume that
	\begin{align*} 
		\rho(v_1^i)&=e_{9(i-1)+1}-e_{9(i-1)+2},\\ 
		\rho(v_2^i)&=e_{9(i-1)+3}-e_{9(i-1)+4},\\
		\rho(v_3^i)&=e_{9(i-1)+4}-e_{9(i-1)+5},\\
		\rho(v_4^i)&=e_{9(i-1)+5}-e_{9(i-1)+6}\\
		\rho(v_6^i)&=e_{9(i-1)+7}-e_{9(i-1)+8},\\
		\rho(v_7^i)&=e_{9(i-1)+8}-e_{9(i-1)+9}.
	\end{align*}
	 Then it follows that the rank of the image of $\rho$ is at least $9n$. Hence $Q_X^n$ cannot be embedded into the standard diagonal lattice of the same rank.
	 
	 Now, let $X$ be the canonical definite 4-manifold of $O_{12(b-2)+5}$. Label by $v_1^i,\dots,v_6^i$, basis vectors of each $i$-th copy of $Q_X^n$, as shown in Figure \ref{fig:Octa}. As usual, an embedding $\rho$ of $Q_X^n$ into the standard diagonal lattice $\langle -1\rangle^r$ should have the form,
	\begin{align*}
		\rho(v_1^i)&=e_{6(i-1)+1}-e_{6(i-1)+2},\\ \rho(v_2^i)&=e_{6(i-1)+3}-e_{6(i-1)+4},\\ \rho(v_3^i)&=e_{6(i-1)+4}-e_{6(i-1)+5},\\ \rho(v_4^i)&=e_{6(i-1)+5}-e_{6(i-1)+6}
	\end{align*} 
	for $i=1,\dots,n$, without loss of generality. Notice that there appeared $6n$ basis vectors in the image of $v_1^i\dots v_4^i$. Thus we only need to show that the image of $v_6^1$ contains a basis vector other than $\{e_1,\dots,e_{6n}\}$. In fact, one can easily see that $\rho(v_6^1)$ should be one of the following forms: 
	\[e_{6n+1}\pm (e_{6(j-1)+1}+e_{6(j-1)+2})\quad\text{or}\quad e_{6n+1}+e_{6n+2}+e_{6n+3}\]
	for some $j\in\{1,\dots,n\}$. Hence the rank of the image of $\rho$ is strictly greater than $6n$.

	A similar argument will be applied to $O_{12(b-2)+7}$. We label by $v_1^i,\cdots,v_5^i,\ i=1,\cdots,n$, the basis vectors of $Q_X^n$ for the plumbed manifold $X$ of $O_{12(b-2)+7}$ as in Figure \ref{fig:Octa}. By the same argument of the previous cases, an embedding $\rho$ of $Q_X^n$ into $\langle -1\rangle^r$ should have the form: 
	\begin{align*}
	\rho(v_1^i)&=e_{5(i-1)+1}-e_{5(i-1)+2}\\
	\rho(v_2^i)&=e_{5(i-1)+3}-e_{5(i-1)+4}\\
	\rho(v_3^i)&=e_{5(i-1)+4}-e_{5(i-1)+5}\\
	\end{align*}
	for $i=1,\cdots,n$. Then one can check that the following are the only choices for $\rho(v_5^1)$:
	\begin{gather*}
	e_{5n+1}+e_{5n+2}+e_{5n+3}+e_{5n+4},\\ 
	e_{5n+1}+e_{5n+2}\pm (e_{5(j-1)+1}+e_{5(j-1)+2}),\\ e_{5n+1}\pm(e_{5(j-1)+3}+e_{5(j-1)+4}+e_{5(j-1)+5}),\text{ or }\\ 
	\pm(e_{5(j-1)+1}+e_{5(j-1)+2})\pm (e_{5(k-1)+1}+e_{5(k-1)+2})
	\end{gather*}
	for some $j, k\in\{1,\cdots,n\}$ such that $j\neq k$. If we admit one of the first $3$ cases, then the rank of the image of $\rho$ is greater than $5n$. If we assume the last case, then $v_5^1\cdot w\equiv (v_1^j+v_1^k)\cdot w$ modulo $2$ for any $w$ in $Q_X^n$. Since $(v_1^j+v_1^k)\cdot v_4^j=(v_1^j+v_1^k)\cdot v_4^k=1$, we have $v_5^1\cdot v_4^j\neq 0$ and $v_5^1\cdot v_4^k\neq 0$, which induces a contradiction.
\end{proof}

\begin{figure}[t!]
	\includegraphics{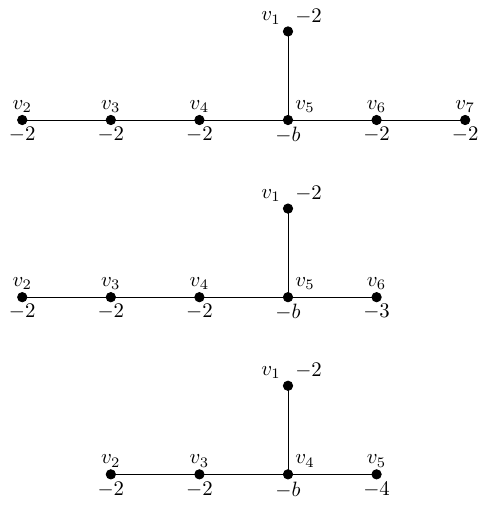}
	\caption{The plumbing graphs for the canonical definite 4-manifolds bounded by $O_{12(b-2)+1}$, $O_{12(b-2)+5}$ and $O_{12(b-2)+7}$, respectively.}
	\label{fig:Octa}
\end{figure}

\subsection{Greene-Jabuka technique}\label{sec:Greene-Jabuka}
The manifolds of the form $I_{30(b-2)+19}$ with odd $b$ is a family of manifolds to which the previous finite order obstructions cannot be applied. The correction terms for spin structures on them vanish, and the intersection lattice of the canonical plumbed manifold of them can be embedded into the standard diagonal one of the same rank as we observed in Section \ref{sec:I_19}. However, a technique of Greene and Jabuka in \cite{Greene-Jabuka:2011-1} turns out to be useful for showing $I_{30(b-2)+19}$ with $b>5$ has infinite order in $\rhcg$. Their main idea is to combine the obstructions from Donaldson's theorem and Heegaard Floer corrections terms for the boundary 3-manifolds of plumbed 4-manifolds. We first recall their theorem.

\begin{theorem}[{\cite[Theorem 3.6]{Greene-Jabuka:2011-1}}]
	\label{thm:Greene-Jabuka}
	Let $G$ be a graph that is a forest of negative definite trees with at most two overweighted vertices. Let $Y$ be the boundary 3-manifold of the plumbed 4-manifold corresponding to $G$. If $Y$ bounds a rational homology ball, then there is a square matrix $A$ such that it factors the incident matrix of $G$, also denoted by $G$, as $G=-AA^T$ and every element of $\coker(A^T)$ has a representative in $\{\pm1\}^n$, where $n$ is the rank of $G$.
\end{theorem}

We remark that the factorization property of the incident matrix $G$ as $G=-AA^T$ is equivalent to the embedding condition of the intersection form of the plumbed 4-manifold $W$ into $\langle-1\rangle^n$. More precisely, if there exists an embedding $\rho$ from $Q_{W}$ into $\langle-1\rangle^n$, let $A^T$ be the matrix representing $\rho$ with respect to the basis of $Q_{W}$ by the spheres corresponding to vertices of $G$ and the standard basis of $\langle-1\rangle^n$. Then one can explicitly check that $A$ satisfies $G=-AA^T$, and vice versa.

On the other hand, the property on the cokernel of $A^T$ comes from the obstruction by Heegaard Floer correction terms and the computation of correction terms for the boundary of a plumbed 4-manifold by Ozsv\'ath and Szab\'o in \cite{Ozsvath-Szabo:2003-1}. In fact a class in $\coker(A^T)$ corresponds to a spin$^c$ structure on $Y$ that extends to $W$, and in order for the corrections term of the spin$^c$ structures to vanish the class contains a representative in $\{\pm1\}^n$. 

We now apply Greene and Jabuka's theorem to show the manifold $I_{30(b-2)+19}$ with $b>5$ has infinite order in $\rhcg$.

\begin{corollary}
	For $b>5$, $I_{30(b-2)+19}$ has infinite order in $\rhcg$.
\end{corollary}
\begin{proof}
	Fix $b\geq2$. Let $Y$ be the manifold $I_{30(b-2)+19}$. Note that $Y$ bounds the canonical negative definite plumbed manifold $W$ corresponding to the graph $G$ in Figure \ref{fig:I_19}. Let $Q$ be the intersection form of $W$ and $\mathcal{Q}$ denote the direct sum of $n$-copies of $Q$. We also denote basis vectors of the $i$-th copy of $Q$ in $\mathcal{Q}$ by $v_1^i, v_2^i,\dots, v_5^i$ following the assignment in Figure \ref{fig:I_19}. Assume that $Y$ is of order $n$ in $\rhcg$.
	
	We first examine necessary conditions on $\mathcal{Q}$. By Donaldson's theorem, there is an embedding $\rho$ from $\mathcal{Q}$ into $\langle -1 \rangle^{5n}$. Let $\{e_1^i,\dots e_{5}^i\}_{i=1}^n$ be the standard basis of $\langle -1 \rangle^{5n}$. Then by the same argument of Proposition \ref{prop:infinite_order_Donaldson}, the images of $\rho$ for $v_1^i$, $v_2^i$ and $v_3^i$ have the following form:
	\[
		\rho(v_1^i)=e_{1}^{i}-e_{2}^{i}, \quad \rho(v_2^i )= e_{3}^{i}-e_{4}^{i}, \quad \rho(v_3^i )= e_{4}^{i}-e_{5}^{i}
	\]
	for $i=1,\dots,n$, up to re-indexing and re-scaling the basis of $\langle -1 \rangle^{5n}$. Note that all $5n$ vectors in the standard basis have already appeared in the image of $v_1^i, v_2^i$ and $v_3^i$. Since $v_4^i\cdot v_4^i=-5$ and $v_4^i\cdot v_k^j=0$ for $k=1,2,3$ and $j=1,\dots,n$, the image of $v_4^i$ has the following form
	\[
	\rho(v_4^i ) = \pm(e_{1}^{\sigma_1(i)}+e_{2}^{\sigma_1(i)})\pm(e_{3}^{\sigma_2(i)}+e_{4}^{\sigma_2(i)}+e_{5}^{\sigma_2(i)})
	\]
	for $i=1,\dots,n$ and some permutations $\sigma_1$ and $\sigma_2$ in $\{1,\dots,n\}$. By re-indexing and re-scaling the basis vectors as $-e_1^i\rightarrow e_2^i$ and $-e_2^i\rightarrow e_1^i$ if necessary, we can further assume that 
	\[
	\rho(v_4^i ) = (e_{1}^{\sigma_1(i)}+e_{2}^{\sigma_1(i)})\pm(e_{3}^{\sigma_2(i)}+e_{4}^{\sigma_2(i)}+e_{5}^{\sigma_2(i)}).
	\] 
	
	Let us re-index the basis vectors of $\mathcal{Q}$ as $w_1^i=v_1^{\sigma_1(i)}$, $w_2^i=v_2^{\sigma_2(i)}$ and  $w_3^i=v_3^{\sigma_2(i)}$, and the basis vectors of $\langle -1 \rangle^{5n}$ as $f_{r}^i=e_{r}^{\sigma_1(i)}$ for $r=1,2$ and $f_{s}^{i}=e_{s}^{\sigma_2(i)}$ for $s=3,4,5$. Then a matrix $A^T$ that represents $\rho$ with respect to the basis $\{w_1^i, w_2^i, w_3^i, v_4^i\}_{i=1}^n\cup\{v_5^i\}_{i=1}^n$ and $\{f_1^i,f_2^i,\cdots,f_{5}^i\}_{i=1}^n$ is given as
	\[A^T=
	\left[ 
	\begin{array}{ccccc}
		B&0&0&0&*\\
		0&B&0&0&*\\
		0&0&\ddots&0&*\\
		0&0&0&B&*\\
	\end{array}
	\right]_{\textstyle ,} 
	\]
	where $B$ is a block matrix of the form either 
	\[
	\left[ 
	\begin{array}{cccc}
		1 & 0 & 0&1\\
		-1 &0&0&1 \\
		0&1&0&1\\
		0&-1&1&1 \\
		0&0&-1&1\\
	\end{array}
	\right] 
	\quad\text{or}\quad
	\left[ 
	\begin{array}{ccccc}
		1 & 0& 0&1 \\
		-1&0&0&1 \\
		0&1&0&-1 \\
		0&-1&1&-1\\
		0&0&-1&-1 
	\end{array}
	\right]_{\textstyle.}
	\]
	Note that the last $n$ columns of $A^T$, denoted by $*$, correspond to the image of $\{v_5^i\}_{i=1}^n$. By the remark after Theorem \ref{thm:Greene-Jabuka} a square matrix $A$ that factors $G$ as $G=AA^T$ can be assumed to have the above form, without loss of generality, since we shall examine the cokernel of $A^T$.
	
	Now we estimate the number of classes in $\coker(A^T)$ that have a representative in $\{\pm1\}^n$. Let $V\vcentcolon=\{(\pm1,\cdots,\pm1)\in \Z^{5n}\}$ and define a relation $\sim$ on $V$ such that $v \sim w$ if and only if $v-w$ is a linear combination over $\Z$ of the first $4n$ column vectors of $A^T$. Define $W\vcentcolon=\{(\pm1,\cdots,\pm1)\in \Z^{5}\}$ and a relation $\sim$ on $W$ such that $v \sim w $ if and only if $v-w$ is a linear combination over $\Z$ of the column vectors of $B$. Then since the initial $4n$ columns of $A^T$ has the  block matrix form by $B$, $|V/{\sim}|=|W/{\sim}|^n$. Let $l\colon\Z^5\rightarrow\Z$ be the function defined as 
	\[l(v)=3(x_1+x_2)\pm 2(x_3+x_4+x_5)\]
	for $v=(x_1,x_2,x_3,x_4,x_5)\in\Z^5$, then $\text{ker}(l)$ is exactly the subspace generated by the column vectors of $B$ ($\pm$ sign of $l$ follows from the choices of $B$). One can explicitly compute that
	\[
	\text{Im}(l|_{W})=\{-12,-8,-6,-4,-2,0,2,4,6,8,12\}.
	\]
	Hence $|W/{\sim}|=|\text{Im}(l|_{W})|=11$ and so $|V/{\sim}|=11^n$. Therefore, there are at most $11^n$ classes in $\coker(A^T)$ that have a representative in $V$. Hence the number of extendable spin$^c$ structures on $\#^nY$ that admit vanishing correction terms is also bounded above by $11^n$. One the other hand, by the correction term obstruction, if $Y$ is of order $n$ in $\rhcg$, there are at least $\sqrt{(30(b-2)+19)^n}$ spin$^c$ structures on $\#^n Y$ with vanishing correction terms. Therefore, if $Y$ has the order $n$ in $\rhcg$, then $b\leq 5$.
\end{proof}

\subsection{Correction terms revisited}\label{sec:d_revisited}
In this section, we finally discuss the following $11$ manifolds whose order has not been determined so far.
\begin{equation*}
	\begin{cases}
		O_{12(b-2)+11}&\text{with}\quad b=2,3,4,5,\\
		I_{30(b-2)+17} &\text{with}\quad b=2,4, 6,\\
		I_{30(b-2)+19} &\text{with}\quad b=5,\\
		I_{30(b-2)+23}&\text{with}\quad b=3,5,7.\\
	\end{cases}
\end{equation*}
We claim that these manifolds also have infinite order in $\rhcg$. The main ingredients of our proof are the correction term obstruction together with the study of metabolizer subgroups in products of finite abelian groups by Kim and Livingston in \cite{Kim-Livingston:2014-1}. We also take advantage of that it is possible to compute all the correction terms explicitly since we only have finitely many manifolds at hand. We remark that the generalizations of our finite order obstructions, Proposition \ref{prop:d_sum} and \ref{prop:I_77}, can be found in  \cite{Grigsby-Ruberman-Strle:2008-1}.

\begin{proposition}\label{prop:d_sum}
	Let $Y$ be a rational homology 3-sphere with $|H_1(Y;\Z)|$  square free. Then if $Y$ has finite order in $\rhcg$, then
	\begin{equation*}
	\sum_{\mathfrak{s}\in Spin^c(Y)}d(Y,\mathfrak{s})=0.
	\end{equation*}
\end{proposition}
\begin{proof}
	Let $Y$ be a rational homology 3-sphere with square free $m=|H_1(Y;\Z)|$. Suppose $\#^{2k}Y$ bounds a rational homology 4-ball. Then there exists a spin$^c$ structure $\mathfrak{t}$ on $\#^{2k}Y$ and a metabolizer $\mathcal{M}$ in $H^2(\#^{2k}Y;\Z)\cong\oplus_{i=1}^{2k}H^2(Y_i;\Z)\cong(\Z_m)^{2k}$ such that \[d(\#^{2k}Y,\mathfrak{t}+\alpha)=0\] for each $\alpha\in \mathcal{M}$. Note that since $m$ is square free, the projection map from the metabolizer $\mathcal{M}$ to each summand $H^2(Y_i;\Z)\cong\Z_m$ is a surjective homomorphism by \cite[Corollary 3]{Kim-Livingston:2014-1}. Let $\alpha_i$ denote the restriction of $\alpha\in \mathcal{M}$ on each $H^2(Y_i;\Z)$. Then by summing all the correction terms for spin$^c$ structures extending to the rational ball, we have
	\begin{align*}
		0&=\sum_{\alpha\in \mathcal{M}}d(\#^{2k}Y,\mathfrak{t}+\alpha)=\sum_{\alpha\in \mathcal{M}}\sum_{i=1}^{2k}d(Y_i,\mathfrak{t}|_{Y_i}+\alpha_i)=\sum_{i=1}^{2k}\sum_{\alpha\in \mathcal{M}}d(Y_i,\mathfrak{t}|_{Y_i}+\alpha_i)\\
		&=\sum_{i=1}^{2k}m^{k-1}\sum_{\mathfrak{s}\in Spin^c(Y_i)}d(Y_i,\mathfrak{s})=2km^{k-1}\sum_{\mathfrak{s}\in Spin^c(Y_i)}d(Y_i,\mathfrak{s}).\\
	\end{align*}
\end{proof}
\begin{corollary}
	The following manifolds have infinite order in $\rhcg$:
	\begin{itemize}
		\item $O_{12(b-2)+11}$ for $b=2,3,4,5$,
		\item $I_{30(b-2)+17}$ for $b=2,6$,
		\item $I_{30(b-2)+19}$ for $b=5$,
		\item $I_{30(b-2)+23}$ for $b=3,5,7$.
	\end{itemize}
\end{corollary}
\begin{proof}
	Observe that each manifold listed above admits square free order of $H_1$. Then by a direct computation of correction terms using the formula (\ref{eq:Ni-Wu}) of Ni and Wu, we find each manifold admits a nonzero sum of correction terms over all spin$^c$ structures on it. Hence they have infinite order in $\rhcg$ by Proposition \ref{prop:d_sum}.
\end{proof}

Now, we are left with only one manifold, $I_{77}$. Notice that the sum of all correction terms on $I_{77}$ is zero. The order of $I_{77}$ is answered by the following finite order obstruction. 

\begin{proposition}\label{prop:I_77}
	Let $Y$ be a rational homology 3-sphere with $|H_1(Y;\Z)|=\Z_{p}\oplus\Z_{q}$, where $p$ is a positive prime, $(p,q)=1$, and $pq$ is odd. Let $\mathfrak{s}_0$ be the unique spin$^c$ structure whose first Chern class is trivial. Consider the subset in $Spin^c(Y)$ defined as \[A=\{\mathfrak{s}_0+n\cdot q\mid n=0,\dots, p-1\}.\] Then, if $Y$ has finite order in $\rhcg$, then \[\max_{\mathfrak{s}\in A}\{d(Y,\mathfrak{s})\}+\min_{\mathfrak{s}\in A}\{d(Y,\mathfrak{s})\}=0.\]
\end{proposition}
\begin{proof}
	Let $Y$ be a rational homology sphere of finite order in $\rhcg$, satisfying the hypothesis in the proposition. Without loss of generality, we may assume that there is $k>0$ such that $\#^{4k}Y$ bounds a rational homology ball. Let $\mathcal{M}$ be a metabolizer of the linking form $\lambda$ of $\#^{4k}Y$ in $\oplus^{4k}H^2(Y;\Z)\cong(\Z_{p})^{4k}\oplus(\Z_{q})^{4k}$ of which the corresponding correction terms vanish. By the classification of linking forms over finite abelian groups \cite{Kawauchi-Kojima:1980-1}, the linking form $\lambda$ can be also decomposed into the direct sum of linking forms on each summand of $(\Z_{p})^{4k}\oplus(\Z_{q})^{4k}$.  Then, by a result of Kim and Livingston \cite[Theorem 4]{Kim-Livingston:2014-1}, $\mathcal{M}$ contains an element of the form \[\gamma=(1,\dots,1,a_{2k+1},\dots,a_{4k})\oplus \beta\in(\Z_{p})^{4k}\oplus(\Z_{q})^{4k}\] for some $a_i\in\Z_p$ and some $\beta\in(\Z_q)^{4k}$. For any $m\in\{0,\dots ,p-1\}$, 
	\begin{align*}
		0&=d(\#^{4k}Y, \#^{4k}\mathfrak{s}_0+mq\cdot(\gamma\oplus\beta))\\
		&=d(\#^{4k}Y, \#^{4k}\mathfrak{s}_0+mq\cdot\gamma\oplus 0)\\
		&=2k\cdot d(Y,\mathfrak{s}_0+mq)+\sum_{i=2k+1}^{4k}d(Y,\mathfrak{s}_0+mqa_i).\\
	\end{align*}
	Note that the first equation follows from Theorem \ref{thm:OS}, and in the third equation we use $d(Y,\mathfrak{s}_0)=0$ from Proposition \ref{prop:infinite}. Now, choose $m$ so that $d(Y,\mathfrak{s}_0+mq)$ is the maximum in $A$. Note that if $d(Y,\mathfrak{s}_0+mq)$ is less than $0$, this contradicts the equation above. Hence we may assume $d(Y,\mathfrak{s}_0+mq)$ is nonnegative. In order for the above equation to hold, the minimum value in $\{d(Y,\mathfrak{s}_0+nq)|n=0,\dots p\}$ must be the negative of the maximum value.
\end{proof}
\begin{corollary}
	The manifold of type $I_{77}$ has infinite order in $\rhcg$.
\end{corollary}
\begin{proof}
	By taking $p=7$ and $q=11$, the manifold $I_{77}$ satisfies the hypothesis of the above proposition. The correction terms of $\{\mathfrak{s}_0+11n|n=0,\dots,6\}$ are given as:
	\begin{equation*}
		\left\{ 0, -\frac{2}{7}, \frac{6}{7}, -\frac{4}{7}, -\frac{4}{7}, \frac{6}{7}, -\frac{2}{7}\right\}	
	\end{equation*}
	by the computation using that $I_{77}\cong S_{\frac{77}{12}}(T_{2,3})$, and Ni and Wu's formula (\ref{eq:Ni-Wu}). Since $\frac{6}{7}-\frac{4}{7}\neq0$, $I_{77}$ has infinite order in $\rhcg$ by Proposition \ref{prop:I_77}.
\end{proof}

\begin{remark}
	We remark that the order of $I_{77}$ is also obtained by a recent result of Aceto and Golla \cite[Theorem 1.1]{Golla-Larson:2018-1}
\end{remark}

\subsection{Concordance of Montesinos knots admitting spherical branched cover} 
We now prove by-products of our results, Corollary \ref{cor:slice-ribbon} and \ref{cor:concordance_order}, about the concordance of the family of Montesinos knots admitting spherical branched double covers. Let $\mathcal{S}$ be the set of such Montesinos knots. The corollaries can be considered as generalizations of the results of Lisca, \cite[Corollary 1.3]{Lisca:2007-1} and \cite[Corollary 1.3]{Lisca:2007-2}, for 2-bridge knots (whose branched covers are lens spaces).   

\begin{proof}[Proof of Corollary \ref{cor:slice-ribbon}]
	By the classifications of Seifert manifolds \cite{Seifert:1933} and Montesinos links \cite{Zieschang:1984-1}, the Montesinos links which admit spherical manifolds as the branched cover are exactly those corresponding to the natural Montesinos branch sets of spherical manifolds. Then, by the following clear implication, 
	\[
	K \text{ is ribbon} \Rightarrow K \text{ is slice} \Rightarrow \Sigma(K) \text{ bounds a rational ball},
	\]
	the corollary is proved if we show that all Montesinos knots $K$ in $\mathcal{S}$ such that $\Sigma(K)$ bounds a rational ball are ribbon. This follows from Proposition \ref{prop:Lecuona_ribbon} for the branch sets of $\mathbf{D}$-type manifolds of order $1$, $T_{3}$, and $T_{27}$, and from Proposition \ref{prop:ribbon} for the branch set of $I_{49}$.
\end{proof}

\begin{proof}[Proof Corollary \ref{cor:concordance_order}]
	Recall that the concordance order of a knot $K$ is bounded below by the cobordism order of $\Sigma(K)$. In the proof of Corollary \ref{cor:slice-ribbon} above, we show that for any knot $K$ in $\mathcal{S}$, $\Sigma(K)$ bounds a rational homology ball if and only if $K$ is slice (of concordance order 1). Thus we only need to show that the branch set of order $2$ spherical manifolds have the same order in the concordance group.
	
	For the lens spaces, this is a result of Lisca \cite[Corollary 1.3]{Lisca:2007-2}. Since the order of $H_1$ of $\mathbf{D}$-type manifolds are even, they cannot be obtained by double cover along a knot. Hence it remains to show that the Montesinos knot $M(4;(2,1),(3,2),(3,1))$, the canonical branch set of $T_{15}$, has concordance order two. In fact this Montesinos knot is isotopic to the knot $9_{24}$ in terms of the notation of the Rolfsen's table \cite{Rolfsen:1976-1}. See \cite[Appendix F.2]{Kawauchi:1996-1} for the fact that $9_{24}=M(4;(2,1),(3,2),(3,1))$. According to KnotInfo of Cha and Livingston \cite{Cha-Livingston:KnotInfo}, the knot $9_{24}$ has concordance order two.
\end{proof}

As a final remark, we note that our work did not employ any particular facts of spherical geometry, but make use of properties of Seifert manifolds and knot surgery manifolds. Thus it is natural to expect generalizations of our results to larger families of manifolds, and we address the following question. 
\begin{question}
	Which Dehn surgeries on the trefoil knot, or more generally 3-legged Seifert manifolds bound rational homology balls?
\end{question}
\begin{remark}
	We remark that Aceto and Golla classified $S^3_{p/q}(T)$ bounding rational balls provided that $p\equiv1$ modulo $q$ \cite[Corollary 4.14]{Aceto-Golla:2017-1}.
\end{remark}

\subsection*{Acknowledgments} 
The authors thank Hakho Choi, Min Hoon Kim, Jongil Park, Motoo Tange and Ki-Heon Yun for their interests in this project and valuable conversations. The second author is partially supported by Basic Science Research Program through the National Research Foundation of Korea (NRF, F2018R1C1B6008364) and BK21 PLUS SNU Mathematical Sciences Division.

\bibliographystyle{alpha}
\bibliography{references}{}
\end{document}